\documentclass[a4paper,11pt]{article}
\usepackage[utf8]{inputenc}
\usepackage{amssymb,amsmath,mathrsfs,amsthm,bbm}
\usepackage{verbatim}%for comment
\usepackage{indentfirst}
\usepackage{geometry}
\usepackage{xcolor}%for changing color
\usepackage[titletoc]{appendix}

\usepackage{hyperref} %navigation
\hypersetup{colorlinks=true,allcolors=blue}
\usepackage{hypcap}

\usepackage{graphicx}
\graphicspath{ {./images/} }
 
\newtheorem{thm}{Theorem}[section]
\newtheorem{defi}[thm]{Definition}
\newtheorem{prop}[thm]{Proposition}
\newtheorem{cor}[thm]{Corollary}

\newtheorem{lem}[thm]{Lemma}
\newtheorem{rem}[thm]{Remark}

\renewcommand{\bf}[1]{\mathbf{#1}}
\renewcommand{\rm}[1]{\mathrm{#1}}
\renewcommand{\cal}[1]{\mathcal{#1}}
\newcommand{\bb}[1]{\mathbb{#1}}

\newcommand{\dd}{\mathrm{d}}

\newcommand{\ov}{\overline}
\newcommand{\dis}{\operatorname{dist}}

\def\R{\bb R}
\def\C{\bb C}
\def\P{\bb P}

\def\B{\mathbbm 1}
\def\H{\bb H}

\newcommand{\wa}{\rightsquigarrow}
\newcommand{\qar}{\rightsquigarrow }
\newcommand{\lar}{\leftrightarrow}

\newcommand{\CG}{C_{\Gamma}}
\newcommand{\lt}{\mathcal{L}}

\newcommand{\ze}{\zeta_{1,\bf A}(\bf b_1)}
\newcommand{\zee}{\zeta_{k,\bf A}(\bf b_k)}
\renewcommand{\Re}{\mathrm{Re}}
\renewcommand{\Im}{\mathrm{Im}}

\def\slc{\rm{SL}_2(\bb C)}

\def\sp{\rm{Span}}

\begin{document}
\title{Kleinian Schottky groups, Patterson-Sullivan measures, and Fourier decay\\}
\author{Jialun Li, Fr\'ed\'eric Naud and Wenyu Pan (with appendix by Jialun Li)\footnote{FN is supported by Institut Universitaire de France.}}
\date{}

%\dedicatory{With an appendix by Jialun Li}
\maketitle
%\author[J. Li]{Jialun Li}
%\address{}
%\email{}

%\author[W. Pan]{Wenyu Pan}
%\address{}
%\email{}

%\subjclass[2010]{Primary:}
%\keywords{}
\begin{abstract} Let $\Gamma$ be a Zariski dense Kleinian Schottky subgroup of $\mathrm{PSL}_2(\C)$. Let $\Lambda_{\Gamma}\subset \C$
be its limit set, endowed with a Patterson-Sullivan measure $\mu$ supported on $\Lambda_{\Gamma}$. We show that the Fourier transform $\widehat{\mu}(\xi)$
enjoys polynomial decay as $\vert \xi \vert$ goes to infinity. This is a $\mathrm{PSL}_2(\C)$ version of the result of Bourgain-Dyatlov \cite{BD1}, and uses the decay of exponential sums based on Bourgain-Gamburd sum-product estimate on $\C$. These bounds on exponential sums require a delicate non-concentration hypothesis which is proved using some representation theory and regularity estimates for stationary measures of certain random walks on linear groups.
\end{abstract}
%\tableofcontents

\bibliographystyle{plain}

\section{Introduction and main result}
\subsection{Fourier dimension}
Let $\mu$ be a Borel probability measure on $\R^d$,  then its Fourier transform $\widehat{\mu}(\xi)$ is defined for any $\xi \in \R^d$ by
$$ \widehat{\mu}(\xi):=\int_{\R^d}e^{-i\langle \xi,x\rangle}d\mu(x).$$
Here $\langle\bullet, \bullet\rangle$ is the usual scalar product on $\R^d$ and $\vert \bullet \vert$ is the associated euclidean norm.
Let $K$ be a non empty compact subset of $\R^d$, then following Frostman \cite{Frost1} its Hausdorff dimension can be expressed as 
$$\mathrm{dim}_H(K)=\sup\left\{ s\in [0,d]\ :\ \int_{\R^d} \vert \widehat{\mu}(\xi)\vert^2\vert \xi\vert^{s-d}d\xi<\infty\ \mathrm{for\ some}\ \mu \in \mathcal{P}(K) \right\}, $$
where $\mathcal{P}(K)$ is the set of Borel probability measures on $K$. On the other hand, the {\it Fourier dimension} is defined by
$$\mathrm{dim}_F(K)=\sup\left\{ s\in [0,d]\ :\ \sup_{\xi}\vert \widehat{\mu}(\xi)\vert^2\vert \xi\vert^{s}<\infty\ \mathrm{for\ some}\ \mu \in \mathcal{P}(K) \right\}.$$ 
We therefore have $\mathrm{dim}_F(K)\leq \mathrm{dim}_H(K)$, and sets for which equality occur are called {\it Salem sets}. Constructing Salem sets
with genuine fractal dimension is a difficult problem, and all the known constructions either rely on the use of a random process \cite{Kahane1, Bluhm1} or specific number theoretic properties \cite{Hambrook1, Kaufman1}. A related problem and still widely open, is to build deterministic sets with {\it positive Fourier dimension}, i.e. compact sets $K$ with fractal Hausdorff dimension
for which one can find a Borel probability measure $\mu$ on $K$ whose Fourier transform has {\it polynomial decay}:
$$\widehat{\mu}(\xi)=O(\vert \xi \vert^{-\epsilon}),$$
for some $\epsilon>0$. This is of course not always possible: in dimension $1$, the celebrated example of the triadic Cantor set is known to have zero Fourier dimension, by the work of Kahane and Salem \cite{KS1}. In higher dimension, any fractal set $K$ which is contained in an {\it affine subspace} will obviously not enjoy Fourier decay as $\vert \xi \vert\rightarrow \infty$.
The problem of Fourier decay of fractal measures is not only interesting for itself but also for its relationship with optics and diffraction through the Huygens-Fresnel principle: 
see for example \cite{AC1, Sakurada1} in the physics literature. In the mathematics literature, in addition to the above mentioned works, Fourier decay is deeply connected to the problem of restriction estimates in harmonic analysis, we just mention \cite{Laba} and references herein. For a comprehensive introduction to fractal sets and the calculation of Hausdorff dimension, we refer to the classic textbook of Falconer \cite{Falconer1}. For an in depth study of the relationships between Hausdorff dimension and Fourier transform, we recommend the book of Mattila \cite{Mattila1}. 

\subsection{Main result}
A recent result of Bourgain-Dyatlov \cite{BD1} shows (in dimension 1) that all limit sets of non-elementary convex co-compact Fuchsian groups have positive Fourier decay, which is an explicit family of examples. Recent works of Sahlsten et al \cite{Sal1,Sal2} and Li \cite{Li1} also prove Fourier decay in deterministic situations (Cantor sets
related to the Gauss map and stationary measures on $\mathrm{SL}_2(\R)$).
Before we state our main theorem,
we need to recall some notations. From now on we take $d=2$ and we identify $\R^2\simeq \C$.  Let $D_1,\ldots,D_r,\ldots ,D_{2r}$ be $2r$ bounded open {\it topological discs}\footnote{In general $\partial D_j$ is just H\"older regular and $D_j$ is not convex.} in $\C$ ,with $r\geq 2$, whose closures are pairwise disjoint:
$$\forall i\neq j,\  \overline{D_j}\cap \overline{D_i}=\emptyset.$$
Assume that we are given $\gamma_1,\ldots,\gamma_r$ in $\mathrm{PSL}_2(\C)$ such that 
$$\forall i=1,\ldots,r,\ \gamma_i(\widehat{\C}\setminus \overline{D_{i+r}})=D_i.$$
Then the free group $\Gamma:=\langle \gamma_1,\ldots,\gamma_r,\gamma_1^{-1},\ldots,\gamma_r^{-1}\rangle$ is called a {\it Schottky group}. If in addition the discs $D_i$
are genuine euclidean discs, then $\Gamma$ is called {\it classical}. The limit set $\Lambda_{\Gamma}$ is the complementary set of the discontinuity set 
$\Omega_\Gamma \subset \widehat{\C}$ for the action of $\Gamma$ on $\widehat{\C}$. When, using Poincar\'e extension, $\Gamma$ is viewed as a set of isometries of the
hyperbolic $3$-space $\H^3$, then $\Lambda_{\Gamma} \subset \partial \H^3$ coincides with the limit set of $\Gamma$ for its action on $\H^3$. The Hausdorff dimension of the limit set
coincides with the critical exponent of Poincar\'e series, and is denoted throughout the paper by $\delta:=\delta_\Gamma$. We point out that there is a universal upper bound strictly smaller than $2$ on the dimension $\delta$ for {\it classical} Schottky groups due to Doyle \cite{Doyle}. On the other hand, non-classical Schottky groups are rather ubiquitous, and free subgroups of co-compact subgroups of $\rm{PSL}_2(\C)$, see L. Bowen \cite{LBowen}, provide examples of non-classical Schottky groups with $\delta$ arbitrarily close to $2$. We recall that limit sets of convex co-compact manifolds are naturally equiped with a measure called {\it Patterson-Sullivan measure}, which equals the $\delta$-Hausdorff measure (with respect to the spherical metric) on $\Lambda_{\Gamma}$ in our setting (see $\S 2$ for more details).
Our main result is as follows.
\begin{thm}
\label{main1}
 Assume that $\Gamma$ is a Zariski dense Schottky group in $\mathrm{PSL}_2(\C)$, and let $\mu$ be a Patterson-Sullivan measure on $\Lambda_{\Gamma}$. Fix any neighborhood $\mathcal{U}$ of $\Lambda_{\Gamma}$.
 Let $g$ be in $C^1(\mathcal{U},\C)$ and $\varphi$ be in $C^2(\mathcal{U},\R)$ with
 $$M:=\inf_{z\in\mathcal{U} }\vert \nabla_z\varphi \vert>0$$
 on $\Lambda_\Gamma$. Assume that $\Vert g \Vert_{C^1}+\Vert \varphi \Vert_{C^2}\leq M'$. Then there exist $C:=C(M,M',\Gamma)>0$ and $\epsilon>0$, with $\epsilon$ depending only on $\mu$, such that
 for all $t\in \R$ with $\vert t\vert\geq 1$, 
 \begin{equation}
 \label{main equation}
 \ \left \vert \int_{\Lambda_\Gamma} e^{it\varphi(z)}g(z)d\mu(z) \right \vert \leq C\vert t \vert^{-\epsilon}.
 \end{equation}
 Moreover, there exists $\alpha>0$ such that if  we have $M\geq \vert t \vert^{-\alpha}$ for all $\vert t\vert$ large, the same conclusion holds.
 \end{thm}
 
 \bigskip
 \noindent
 \begin{rem}
 \label{rema}
 \begin{enumerate}
 \item In the case of $\mathrm{PSL}_2(\mathbb{R})$, Theorem \ref{main1} is obtained by Bourgain-Dyatlov \cite[Theorem 1.2]{BD1} and they show that the decay rate $\epsilon$ depends only on the Hausdorff dimension $\delta_{\Gamma}$. In our setting, 
 	the decay rate $\epsilon$ depends on $\delta_\Gamma$ and the regularity constant $\kappa_8$ given in Lemma \ref{bq-reg}. It is natural to expect that in higher dimensions extra quantities will appear in the characterization of the decay rate. This is because there is no uniform decay rate for any fixed $\delta_\Gamma<1$. Indeed, a Zariski dense Schottky group with $\delta_\Gamma<1$ can be arbitrarily close to a subgroup contained in $\rm{PSL}_2(\R)$, which has no such Fourier decay (see Corollary \ref{cor}). It would be interesting to find a geometric interpretation of the regularity constant $\kappa_8$.
 	
 \item The decay rate $\epsilon$ does not change if we pass to finite index subgroups, because the Patterson-Sullivan measure remains the same when passing to finite index subgroups, a result due to Roblin \cite[Lemma 2.1.4. Theorem 2.2.2]{roblin2005}.
 
 \item Let $\Gamma$ be as in Theorem \ref{main1}. We consider the Selberg zeta function $Z_{\Gamma}(s)$ for the quotient $\Gamma\backslash \mathbb{H}^3$. Using the method originating in the work of Dolgopyat \cite{Dol}, Stoyanov showed that $Z_{\Gamma}$ has finitely many zeros in $\{\operatorname{Re} s\geq \delta_{\Gamma}-\epsilon_{\Gamma}\}$ for some $\epsilon_{\Gamma}>0$ depending on $\Gamma$ \cite{Sto1}. Now with Theorem \ref{main1} available, following the exact same arguments as in \cite{BD1}, we can obtain an $\epsilon_0>0$ depending only on the Fourier decay rate $\epsilon$ given in (\ref{main equation}) such that $Z_{\Gamma}(s)$ has only finitely many zeros in $\{\Re(s)>\delta-\epsilon_0 \}$. In particular, this yields a uniform "essential" spectral gap for any finite cover of $\Gamma\backslash \mathbb{H}^3$ by the above remark. We point out that under the assumption that $\delta_{\Gamma}$ is close enough to $1$, this was already obtained by Dyatlov and Zahl in \cite{DZ}.
Moreover, in dimension 2, Dyatlov and Jin \cite{DJ} also have (unconditionally) an explicit estimate for the spectral gap of convex co-compact surfaces which depends only on the Patterson-Sullivan measure.
 
 \item In \cite{li2018thesis}, the first author established Fourier decay for split semi-simple groups. To deal with the non-split group $\rm{PSL}_2(\C)$, we will borrow ideas from \cite{li2018thesis}, while following a new scheme. We would like to point out that it seems possible to combine the methods in this paper with the ones in \cite{li2018thesis} to obtain a Fourier decay for Furstenberg measures on $\hat{\C}$ but we do not pursue this generalization here. 
 \end{enumerate}
  \end{rem}

\subsection{$C^2$-stable positive Fourier dimension} 
Theorem \ref{main1} motivates the following definition.
 \begin{defi}
 A compact set $K\subset \C$ is said to have  \underline{$C^2$-stable positive Fourier dimension} if and only if for all 
 $C^2$-diffeomorphism $\phi:\mathcal{U}\rightarrow \phi(\mathcal{U})\subset \C$, defined on a neighborhood $\mathcal U$ of $K$, $\phi(K)$ has positive Fourier dimension.
 \end{defi}
Theorem \ref{main1} implies the following characterization of "stable Fourier decay" for limit sets of Schottky groups.
\begin{cor}\label{cor}
Let $\Gamma$ be a Schottky group as above, then $\Lambda_{\Gamma}$ has $C^2$-stable positive Fourier dimension if and only if
$\Gamma$ is Zariski dense in $\rm{PSL}_2(\C)$.
\end{cor}
\noindent {\it Proof.} Assume first that $\Gamma$ is Zariski dense, and denote by $\phi:\mathcal{U}\rightarrow \phi(\mathcal{U})$ an arbitrary $C^2$-diffeomorphism with finite
$C^2$ norm on $\mathcal{U}$,
with $\mathcal {U} \supset \Lambda_{\Gamma}$ an open bounded set. Let $\phi^*\mu$ be the push-forward of a Patterson-Sullivan measure $\mu$,
then
$$\widehat{\phi^*\mu}(\xi)=\int_{\Lambda_{\Gamma}} e^{-i\langle \xi,\phi(z) \rangle}d\mu(z).$$
Set $\xi=t\theta$ where $t>0$ and $\vert \theta \vert=1$, so that we have
$$\widehat{\phi^*\mu}(\xi)=\int_{\Lambda_{\Gamma}} e^{-it\varphi_\theta(z)}d\mu(z), $$
with $\varphi_\theta(z)=\langle \theta, \phi(z)\rangle$. Notice that for all $z=x+iy \in \mathcal{U}$, we have
$$\vert \nabla_z \varphi_\theta \vert^2=(\langle \theta, \partial_x \phi(z)\rangle)^2 +(\langle \theta, \partial_y \phi(z)\rangle)^2.$$
Because $\phi$ is a diffeomorphism, for all $z\in \mathcal{U}$ we get that $\partial_x \phi(z)$ and $\partial_y \phi(z)$ are linearly independent
vectors, which obviously implies that $\nabla_z \varphi_\theta \neq 0$. Because $\Vert \varphi_\theta \Vert_{C^2}$ can be bounded uniformly in $\theta$,
we can apply Theorem \ref{main1} to deduce that for all $\vert \xi \vert\geq 1$, we have
$$\vert \widehat{\phi^*\mu}(\xi)\vert \leq C\vert \xi \vert^{-\epsilon},$$
for some $\epsilon>0$. Hence $\phi(\Lambda_{\Gamma})$ has positive Fourier dimension.

Conversely, assume that $\Gamma$ is {\it not} Zariski dense. Consider the Zariski closure $H$ of $\Gamma$ in $\rm{PSL}_2(\C)$. There is a general fact, see for example \cite{BZ1} and references herein, 
which says that a non-compact proper Lie subgroup of $H\subset \mathrm{Isom}^+(\H^3)$  which has no fixed point for its action on $\partial \H^3$ has an invariant totally geodesic proper submanifold of $\H^3$. Since $\Gamma$ is 
taken non-elementary, $\Gamma$ must therefore leave invariant a circle for its action on $\partial \H^3=\widehat{\C}$. It is not difficult to see then that 
$\Lambda_{\Gamma}$ must be included in that invariant circle. As a consequence, the limit set of $\Gamma$ can be mapped inside the real line $\R$
by a M\"obius map $\phi$. But clearly for any finite Borel measure $\nu$ supported on $\phi(\Lambda_{\Gamma})\subset \R$, we have
$$\widehat{\nu}(\xi)=\int_{\phi(\Lambda_{\Gamma})} e^{-i\langle \xi,z\rangle}d\nu(z)=\nu(\phi(\Lambda_{\Gamma}))\neq 0,$$
whenever $\Re(\xi)=0$. Hence $\phi(\Lambda_{\Gamma})$ has zero Fourier dimension. $\square$

\bigskip
\begin{center}
\includegraphics[scale=0.35]{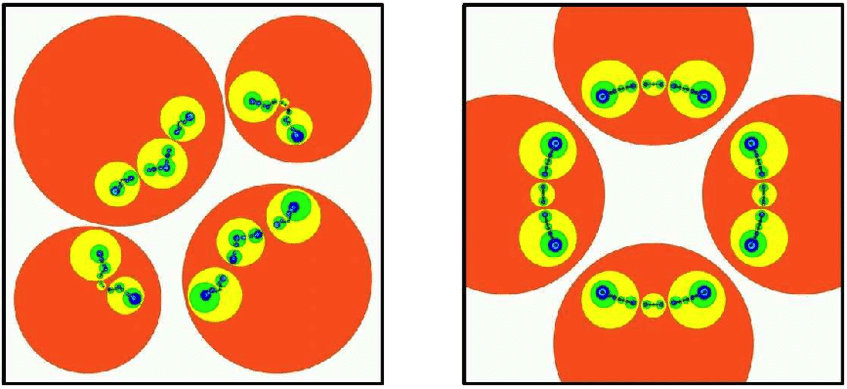}\\
{\small Figure 1: On the left a Zariski dense case, on the right a Fuchsian case, where $C^2$-stable positive Fourier dimension fails.}
\end{center}
 \bigskip \noindent

 \subsection{About the proof of Theorem \ref{main1}}
 Let us now comment on the structure of the proof. After some preliminary facts and notations related to Schottky subgroups of $\mathrm{PSL}_2(\C)$ gathered in $\S2$, we show in $\S 3$ how Theorem \ref{main1} follows from an estimate on decay of exponential sums based on Bourgain-Gamburd sum-product estimate on $\mathbb{C}$, under a {\it non-concentration}
 hypothesis and this generalizes the main ideas of \cite{BD1}. Unfortunately, this non-concentration hypothesis cannot be verified by elementary methods as was done in the $\mathrm{PSL}_2(\R)$ case in \cite{BD1}. We have in particular to check that this non-concentration property holds uniformly "in every direction", which requires the use of some more sophisticated arguments of representation theory and some regularity properties of Patterson-Sullivan measures borrowed from the work on random walks by Guivarc'h \cite{Gui90}. The last section is devoted to the proof of this non-concentration hypothesis which is the main difficulty of the paper. In the appendix,
 the first author proves that Patterson-Sullivan measures arise as stationary measures of certain random walks on $\mathrm{SL}_2(\C)$ with {\it finite exponential moment}, which allows us to use the key regularity property of Guivarc'h.
 
 Verifying non-concentration hypothesis is the main challenge when trying to apply discretized sum-product estimates. For example, in the breakthrough work of Bourgain-Gamburd \cite{BG08}, it is 
 precisely the non-concentration hypothesis that prevents them from obtaining a spectral gap outside of elements with algebraic entries. What's more, in our situation, the Fourier decay is almost equivalent to the non-concentration hypothesis, because the Fourier decay will imply a spectral gap of the transfer operator, which in turn can be used to get the non-concentration hypothesis.
 
 \bigskip \noindent
 {\it Acknowledgements.}
 
 \noindent
We would like to thank Jean-Fran\c cois Quint for pointing out to us the result of Roblin. Fr\'ed\'eric Naud thanks Semyon Dyatlov for stimulating discussions around his work with Jean Bourgain
at the IAS Emerging Topics workshop on "quantum chaos and fractal uncertainty principle" in October 2017. Wenyu Pan would like to thank Hee Oh and Federico Rodriguez-Hertz for their interest and encouragement on the project. 
 
 \section{Preliminary estimates on Schottky groups}
 In this section, we gather notations and important but elementary bounds that will be used in $\S 3$.
\noindent We use similar notations as the ones introduced in the Bourgain-Dyatlov paper \cite{BD1}. Recall that we are given a set of pairwise disjoint open topological discs $D_1,\ldots,D_{2r}$ and we fix a set of generators
$\gamma_1,\ldots,\gamma_r$ in $\mathrm{PSL}_2(\C)$ such that 
$$\forall i=1,\ldots,r,\ \gamma_i(\widehat{\C}\setminus \overline{D_{i+r}})=D_i,$$
see the figure below where $r=2$.
 \begin{center}
\includegraphics[scale=0.35]{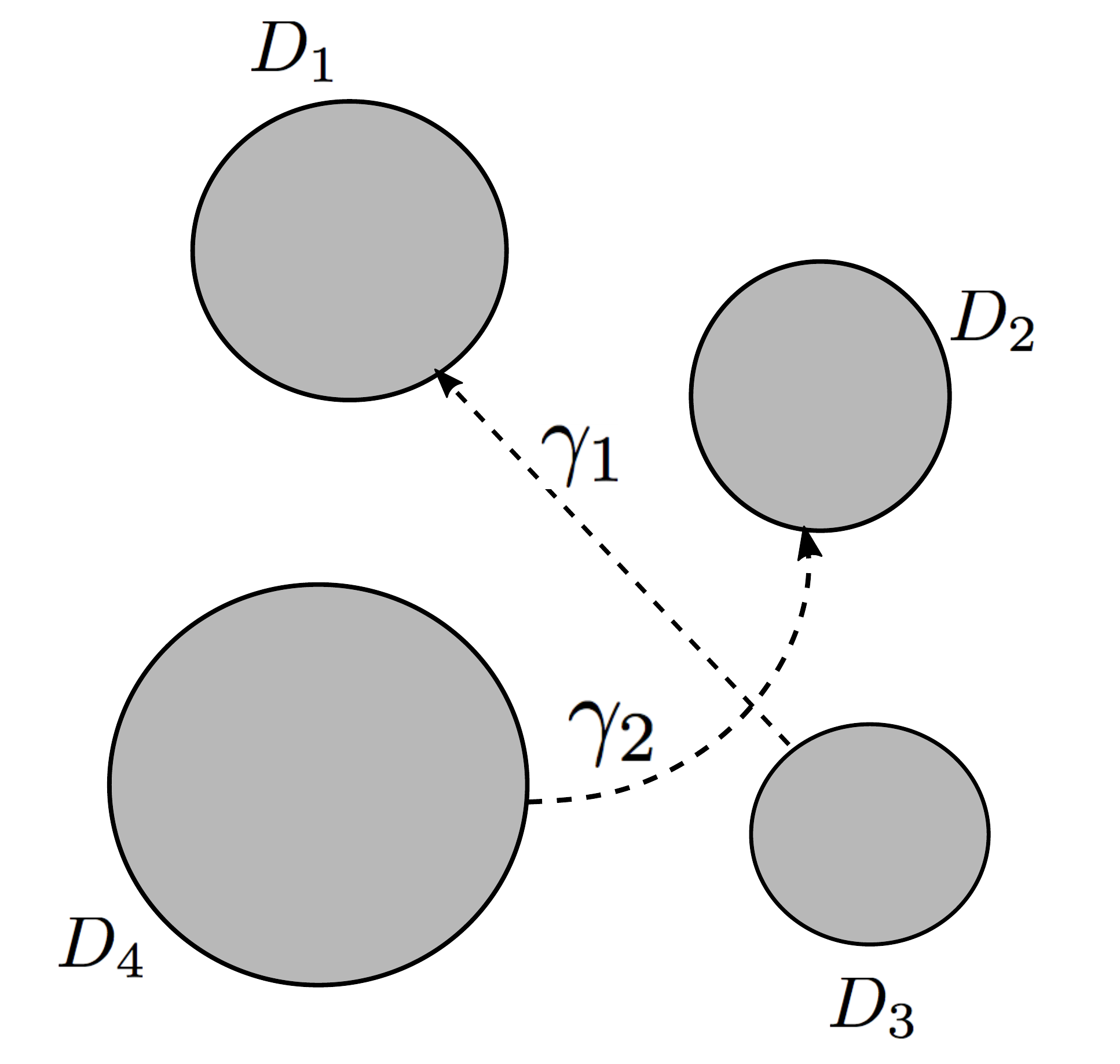}\\
{\small Figure 2: A Schottky pairing.}

\end{center}
For convenience, for $j=r+1,\ldots,2r$, we set $\gamma_j:=\gamma_{j-r}^{-1}$. By the usual ping-pong argument, $\gamma_1,\ldots,\gamma_{2r}$ generate a free group denoted by $\Gamma$ which is convex co-compact.
We will frequently use the notation 
$$\bf D:=\bigcup_{j\in \mathcal{A}} D_j,$$ where the alphabet $\mathcal{A}$ is just the finite set
$$\mathcal{A}:=\{1,\ldots,r,r+1,\ldots,2r \}.$$
Let $\Lambda_\Gamma$ be the limit set of $\Gamma$, defined as the set of accumulation points (in $\partial \H^3=\widehat{\C}$) for the action of $\Gamma$ on $\H^3$.
The action of $\Gamma$ on $\widehat{\C}\setminus \Lambda_\Gamma$ is proper discontinuous and $\widehat{\C}\setminus \bf D$ is a fundamental domain for this action. 
\begin{itemize}
\item For $a\in \mathcal{A}$, we set $\overline{a}:=a+r\  \mathrm{mod}\  2r$ such that $\gamma_{\overline{a}}=\gamma_a^{-1}$.
\item For $n\in\mathbb{N}_0$, define $\mathcal{W}_n$, the set of {\it reduced words} of length $n$, by
\begin{equation*}
\mathcal{W}_n:=\{a_1\ldots a_n| \,\,\,a_1,\ldots, a_n\in \mathcal{A},\,\,\, a_{j+1}\neq \overline{a_{j}}\,\,\,\text{for}\,j=1,\ldots, n-1\}.
\end{equation*}
Denote by $\mathcal{W}:=\cup_n \mathcal{W}_n$ the set of all words. The length of a word $\bf a=a_1\ldots a_n$ is denoted by $\vert \bf a \vert=n$.

Denote the empty word by $\emptyset$ and put $\mathcal{W}^{\circ}:=\mathcal{W}\backslash \{\emptyset\}$. For $\bf a=a_1\ldots a_n\in \mathcal{W}$, put $\bar{\bf a}:=\overline{a_n}\ldots \overline{a_1}\in \mathcal{W}$. 
If $\bf a\in \mathcal{W}^\circ$, put $\bf a':=a_1\ldots a_{n-1}\in \mathcal{W}$. 

\item For $\bf a=a_1\ldots a_n$, $\bf b=b_1\ldots b_m\in \mathcal{W}$, we write $\bf a\to \bf b$ if either at least one of $\bf a,\,\bf b$ is empty or $a_n\neq \overline{b_1}$. Under this condition the concatenation $\bf a \bf b$ is a word.

\item For $\bf a,\bf b\in \mathcal{W}$, we write $\bf a\prec \bf b$ if $\bf a$ is a prefix of $\bf b$, that is $\bf b=\bf a\bf c$ for some $\bf c\in \mathcal{W}$.

\item For $\bf a=a_1\ldots a_n$, $\bf b=b_1\ldots b_m\in \mathcal{W}^{\circ}$, we write $\bf a \qar \bf b$ if $a_n=b_1$. Note that when $\bf a\qar \bf b$, the concatenation $\bf a'\bf b$ is a word of length $n+m-1$.

\item For each $\bf a=a_1\ldots a_n\in \mathcal{W}$, define the group element $\gamma_{\bf a}\in \Gamma$ by
\begin{equation*}
\gamma_{\bf a}:=\gamma_{a_1}\ldots \gamma_{a_n}.
\end{equation*}
Note that each element of $\Gamma$ is equal to $\gamma_{\bf a}$ for a unique choice of $\bf a$ and $\gamma_{\bar{\bf a}}=\gamma_{\bf a}^{-1}$, $\gamma_{\bf a\bf b}=\gamma_{\bf a}\gamma_{\bf b}$ when $\bf a\to \bf b$. 

\item We then define the {\it cylinder sets} associated to reduced words. Given $\bf a=a_1\ldots a_n \in \mathcal{W}^{\circ}$, we set 
$$D_{\bf a}:=\gamma_{\bf a'}(D_{a_n}).$$
Remark that cylinder sets are topological discs, but may not be convex at all in the non-classical case.
\item Given $\gamma \in \Gamma$ we will often write 
$$\gamma \simeq \begin{pmatrix}
	a & b\\ c & d
	\end{pmatrix},$$
meaning that we have :
$$\forall z\in {\C},\  \gamma(z)=\frac{az+b}{cz+d}\ \mathrm{with}\ ad-bc=1.$$

\end{itemize}
Finally, we warn the reader about constants: throughout the rest of this paper $C_\Gamma$ is a positive constant that depends only on $\Gamma$ (more accurately
on the choice of generators as above). This constant $C_\Gamma$ may change from line to line, while still being denoted the same. Given $x,y,C>0$, we denote by
$x \approx_{C} y$ the set of inequalities:
$$ C^{-1} y\leq x \leq C y.$$

The following estimates mimic the ones that are found in \cite{BD1}. However, since we do not work a priori with convex cylinder sets, all diameter estimates are replaced
with measure estimates with respect to the Patterson-Sullivan measure $\mu$, see below.
\subsection{A Lipschitz property}
For the rest of the paper, fix $\epsilon_0>0$ such that $2\epsilon_0>\inf_{j\neq l} d(D_j, D_l)$.

\begin{lem}\label{lem:lip log'}
	There exists $C_{\Gamma}>0$ such that the following holds. For any $\bf a\in \mathcal{W}^{\circ}$ and any $z,w\in \mathbb{C}$ with $d(z, D_{\ov {\bf a}})>\epsilon_0$ and $d(w, D_{\ov{\bf a}})>\epsilon_0$, we have
	\begin{equation}
	\label{lip log eq}
	|\gamma_{\bf a}'z-\gamma_{\bf a}'w|\leq C_{\Gamma}|z-w|(|\gamma_{\bf a}'z||\gamma_{\bf a}'w|)^{1/2}.
	\end{equation}
\end{lem}
\begin{proof}
	Suppose that $\gamma_{\bf a}=\begin{pmatrix}
	a & b\\ c & d
	\end{pmatrix} $. Then 
	\begin{align*}
	|\gamma_{\bf a}' z-\gamma_{\bf a}'w| = &\left|  \frac{1}{(cz+d)^2}-\frac{1}{(cw+d)^2}\right|\\
	=&\left|\frac{c(z-w)}{(cz+d)(cw+d)}\left(\frac{1}{cz+d}+\frac{1}{cw+d}\right)\right| .
		\end{align*}
		Observe that 
	\[\frac{c(z-w)}{cz+d}=\frac{z-w}{z+d/c} \,\,\,\text{and}\,\,\,\frac{c(z-w)}{cw+d}=\frac{z-w}{w+d/c}. \]
	Moreover, we have$-d/c=\gamma_{\bf a}^{-1}(\infty)\in D_{\ov{\bf a}}$ and $d(z, D_{\ov {\bf a}}), d(w, D_{\ov{\bf a}})>\epsilon_0$. These facts imply 
	that \[\left|\frac{z-w}{z+d/c}\right|\leq C_{\Gamma}|z-w|,\,\,\,\left|\frac{z-w}{w+d/c}\right|\leq \CG |z-w| \]
	and the inequality (\ref{lip log eq}) follows.
\end{proof}

We recall the following fundamental formula for M\"obius transformations that will be used throughout the paper.
\begin{lem}
\label{elementary}
For any $\gamma\in \Gamma\backslash \{e\}$ and any $x,y\in \mathbb{C}\backslash \{\gamma^{-1}(\infty)\}$, we have
\begin{equation*}
|\gamma x- \gamma y|=|x-y| |\gamma' x|^{1/2} |\gamma' y|^{1/2}.
\end{equation*}
\end{lem} The proof is by straightforward computation.

\subsection{Facts on Patterson-Sullivan measures}
We refer the reader to \cite{Sullivan1,Sullivan2} for the introduction of Patterson-Sullivan measures.
Let us recall some basic facts of Patterson-Sullivan theory which will be used in this paper. Let $\H^3$ be the upper half-space model of the hyperbolic space, given by
$$\H^3=\C_z\times \R^+_y,$$
endowed with the hyperbolic metric $g$ given by
$$g=\frac{dzd\overline{z}+dy^2}{y^2}.$$
We will fix a base point $o:=(0,1)\in \H^3$. Let $\Gamma$ be a convex co-compact group of isometries of $\H^3$, for example a Schottky group as defined previously.
For all $s>\delta_\Gamma$ and $x\in \H^3$, one sets
$$\mu_x^s:=\frac{1}{P_\Gamma(o,s)} \sum_{\gamma \in \Gamma} e^{-sd(x,\gamma (o))} \mathcal{D}_{\gamma o},$$
where $P_\Gamma(o,s)$ is the convergent Poincar\'e series
$$P_\Gamma(o,s):= \sum_{\gamma \in \Gamma} e^{-sd(o,\gamma (o))},$$
and $\mathcal{D}_x$ is the dirac mass at $x$. The distance $d(x,y)$ above is with respect to the hyperbolic metric. By taking {\it weak limits} of these measures as $s\rightarrow \delta_\Gamma$, one obtains a family of measures supported on the limit set (called Patterson-Sullivan measures)
$\Lambda_\Gamma$ satisfying the following properties:
\begin{itemize}
\item For all $\gamma \in \Gamma$, $\gamma^* \mu_x=\mu_{\gamma^{-1}x}$.
\item For all $x,x'$, we have $\mu_{x'}=e^{-\delta B_\xi(x',x)}\mu_x,$ where $B_\xi(x',x)$ is the Busemann cocycle defined by (here $\xi \in \partial \H^3$)
$$B_\xi(x,y)=\lim_{z\rightarrow \xi} (d(x,z)-d(y,z)).$$ The Busemann cocycle is a smooth function that can be expressed in terms of Poisson kernels.
\end{itemize}
An important fact is that given an isometry $\gamma$, we have 
$$e^{-B_\xi(\gamma^{-1}o,o)}=\vert \gamma'(\xi)\vert_{\bb S^2},$$
where $|\gamma'(\xi)|_{\bb S^2}$ is the derivative of $\gamma$ at $\xi$ for the {\it spherical metric} on $\C\cup\infty$.  In particular, Patterson-Sullivan measures satisfy the equivariant formula
\begin{equation}\label{equ:quasi-invariant}
\forall \gamma \in \Gamma,\ \gamma^*\mu_x=e^{-\delta B_\xi (\gamma^{-1}x,x)}\mu_x.
\end{equation}

Because these measures $\mu_x$ are all absolutely continuous with respect to each other, we will focus on $\mu:=\mu_o$ and refer to it as the "Patterson-Sullivan measure" on $\Lambda_\Gamma$.
Remark that given the above definition, it is a probability measure. Under the action of $\Gamma$, we have therefore the following key formula: for all bounded Borel function $f$ on $\C$ and $\gamma$ in $\Gamma$
\begin{equation}
\label{equivariance}
\int_{\Lambda_\Gamma}f(z)\dd\mu(z)=\int_{\Lambda_\Gamma}f(\gamma z)|\gamma'(z)|_{\bb S^2}^\delta\dd\mu(z). 
\end{equation}
The spherical metric on $\C\cup\{\infty \}$ can be written as $\frac{4|\dd z|^2}{(1+|z|^2)^2} \text{ for }z\in\C\cup \{\infty \}.$
Hence,
\[|\gamma'(z)|_{\bb S^2}=\frac{1+|z|^2}{1+|\gamma z|^2}|\gamma'(z)|. \]
For $\bf a\in \cal W$, we will use the notation
\begin{equation}\label{equ:waz}
w_{\bf a}(z):=|\gamma_{\bf a}'(z)|_{\bb S^2}^\delta.
\end{equation}

\subsection{Distortion estimates for M\"{o}bius transformations}
Let $\Gamma$ be a Schottky group as above. For $\gamma\simeq \begin{pmatrix}
a & b \\  c & d
\end{pmatrix}$, set $\|\gamma\|_E:=\sqrt{a^2+b^2+c^2+d^2}$ and $\|\gamma\|_{\mathcal{S}}:=|c|$. 
\begin{lem}\label{lem:norm c} There exists $C_\Gamma>0$ such that for all $\gamma$ in $\Gamma\backslash \{ e\}$, 
	we have $\|\gamma\|_{\mathcal{S}}\approx_{C_{\Gamma}}\|\gamma\|_{E}$.
\end{lem}
\begin{proof}
         We will use the fact that $\widehat{\C}\setminus \bf D$ is a fundamental domain for the action of $\Gamma$ on $\widehat{\C}\setminus \Lambda_\Gamma$.
         In particular, if we have $\gamma \in \Gamma\setminus\{ e\}$ and $z\in \widehat{\C}\setminus \overline{\bf D}$, then
          $\gamma(z)\in \bf D$.
         First start with $C_{\Gamma}$ to be 
         $$C_\Gamma:=\max_{j\in \mathcal{A}}\{ \sup_{z\in D_j}|z|\}.$$
	  The bound $\|\gamma\|_E\geq |c|$ is trivial. 
	  Now pick any $\gamma\simeq \begin{pmatrix}
a & b \\  c & d
\end{pmatrix}\in \Gamma\backslash \{e\}$.	
         Since $\infty$ is not contained in $\overline{\bf D}$, we have therefore $\gamma(\infty),\gamma^{-1}(\infty)\in \bf D$. Hence $c\neq 0$ and 
	\[ C_{\Gamma}\geq |\gamma(\infty)|=|a/c|, \,\,\,C_{\Gamma}\geq |\gamma^{-1}(\infty)|=|d/c|.  \]
	These imply $|a|,|d|\leq C_{\Gamma}|c|$. 
	
	Now we can bound $|b|$. Observe that  one of the points $\gamma(0),\gamma^{-1}(0)$ must be in $\overline{\bf D}$. Otherwise, we have two points $z=\gamma(0)$ and $w=\gamma^{-1}(0)$ outside of $\overline{\bf D}$, but $\gamma^2w=z$. This forces $\gamma^2$ to be the identity. But $\Gamma$ is a free group, therefore
	$\gamma$ is also the identity. A contradiction. Hence
	\[\text{either}\,\,\,C_{\Gamma}\geq |\gamma(0)|=|b/d| \,\,\,\text{ or }\,\,\,C_{\Gamma}\geq |\gamma^{-1}(0)|=|b/a|.\]
This yields	
	\[\text{either}\,\,\, |b|\leq C_{\Gamma}|d|\,\,\,\text{ or}\,\,\,|b|\leq C_{\Gamma}|a|. \]
	Therefore, we have $\|\gamma\|_E\leq C_\Gamma' |c|$.
\end{proof}

\begin{lem}\label{lem:lip derivative}
	There exists $C_{\Gamma}>0$ such that for all $b\in \mathcal{A}$, all $x\in D_b$ and all word $\bf a$ with $\bf a\qar b$, we have
	\begin{align*}
	C_{\Gamma}^{-1}\|\gamma_{\bf a'}\|_{\mathcal{S}}^{-2}\leq |\gamma'_{\bf a'}x|\leq C_{\Gamma}\|\gamma_{\bf a'}\|_{\mathcal{S}}^{-2}.
	\end{align*}
%Moreover, we have $|\gamma_{\bf a'}'x|$ is $C^{-1}_{\Gamma}\|\gamma_{\bf a'}\|_{\mathcal{S}}^{-2}$ Lipschitz on $D_b$ and $\log|\gamma'_{\bf a'}x|$ is $C^{-1}_{\Gamma}$ Lipschitz on $D_b$.
\end{lem}
\begin{proof}
	Suppose that $\gamma_{\bf a'}\simeq \begin{pmatrix}
	a & b \\ c & d
	\end{pmatrix}\in \slc$. Then $\gamma'_{\bf a'}x=\frac{1}{c^2(x+d/c)^2}$.  As $x\in D_b$, we have that $|x+d/c|=|x-\gamma^{-1}_{\bf a'}(\infty)|\geq 1/\CG$. Meanwhile, we have $x,\gamma^{-1}_{\bf a'}(\infty)\in \bf D$. Hence $|\gamma'_{\bf a'}x|\in [1/\CG,\CG]\|\gamma\|_{\mathcal{S}}^{-2}$.
	
	%The Lipschitz property of $|\gamma'x|$ can be obtained from Lemma \ref{lem:lip log'}. In fact, we have
	%\begin{align*}
	%\left ||\gamma'z|-|\gamma'w|\right|\leq |\gamma'z-\gamma'w|\leq \CG|z-w|C\|\gamma\|_{\mathcal{S}}^{-2}.
	%\end{align*}
	
	%For the Lipschitz property of $\log|\gamma'x|$, suppose that $|\gamma'z|\geq |\gamma'w|$. We have
	%\begin{align*}
	%\left |\log|\gamma'z|-\log|\gamma'w|\right|=\log\frac{|\gamma'z|}{|\gamma'w|}=&\log\left (1+\frac{|\gamma'z|-|\gamma'w|}{|\gamma'w|}\right)\leq \frac{|\gamma'z|-|\gamma'w|}{|\gamma'w|}\\
	%\leq &\CG|z-w|,
	%\end{align*}
	%where the last inequality follows from Lemma \ref{lem:lip log'}.
\end{proof}
\begin{lem}\label{lem:ia gammaa}
For any $\bf a\in \mathcal{W}^{\circ}$, we have
\begin{equation*}
C_{\Gamma}^{-1} \mu(D_{\bf a})\leq |\gamma'_{\bf a'}(x)|^{\delta}\leq C_{\Gamma}\mu(D_{\bf a})\,\,\,\text{for any}\,\,\, x\in D_{a_n}.
\end{equation*}
\end{lem}
\begin{proof}
	Due to Lemma \ref{lem:lip derivative}, it suffices to show $\mu(D_{\bf a})\approx_{C_{\Gamma}} \|\gamma_{\bf a'}\|^{-2\delta}_{\mathcal{S}}$.  We have
\begin{equation*}
\mu(D_{\bf a})=\int_{D_{a_n}} w_{\bf a'}(x) d\mu(x).
\end{equation*}
By Lemma  \ref{lem:lip derivative}, we have
\begin{equation*}
C_{\Gamma}^{-1} \|\gamma_{\bf a'} \| ^{-2\delta}_{\mathcal{S}}\leq w_{\bf a'}(x)\leq C_{\Gamma} \|\gamma_{\bf a'} \| ^{-2\delta}_{\mathcal S}\,\,\,\text{on}\,\,\, D_{a_n}
\end{equation*}
and Lemma \ref{lem:ia gammaa} follows.
\end{proof}

\subsection{More distortion estimates}
\label{distortion-estimates}
\begin{lem}
	\label{lem:contraction}
We have the following contraction property: for any $\bf a\in \mathcal{W}^{\circ},\,b\in \mathcal{A},\,\bf a\to b$, we have
\begin{equation}
\mu(D_{\bf a b})\leq (1-C_{\Gamma}^{-1})\mu(D_{\bf a}).
\end{equation}
\end{lem}

\begin{proof}
We have
\begin{equation}
\mu(D_{\bf a}\backslash D_{\bf a b})=\int_{D_{a_n}\backslash D_{a_n b}} w_{\bf a'}(x)d\mu(x)\geq C_{\Gamma}^{-1}\mu(D_{\bf a})\mu(D_{a_n}\backslash D_{a_n b}),
\end{equation}
where we use Lemma \ref{lem:ia gammaa} to obtain the inequality on the right. As we have a uniform non-trivial lower bound for the measure of the sets of the form $D_{a_n}\backslash D_{a_nb}$, the proof of Lemma \ref{lem:contraction} is complete.
\end{proof}

\begin{lem}[Parent-child ratio]
\label{lem-pa-ch}
For any $\bf a \in \mathcal{W}^{\circ},\,b\in \mathcal{A},\,\bf a\to b$, we have
\begin{equation}
\label{pa-ch}
C_{\Gamma}^{-1} \mu(D_{\bf a})\leq \mu(D_{\bf a b})\leq \mu(D_{\bf a}).
\end{equation}
\end{lem}

\begin{proof}
We just need to show the lower bound. We have
\begin{equation}
\label{eq-pa-ch}
\mu(D_{\bf a b})=\int_{D_{a_n b}} w_{\bf a'}(x)d\mu(x)\geq C_{\Gamma}^{-1}\mu(D_{\bf a})\mu (D_{a_n b}),
\end{equation}
where we use Lemma  \ref{lem:ia gammaa} to obtain the inequality on the right. As we have a uniform non-trivial lower bound for the measure of the sets of the form $\mu (D_{a_n b})$, (\ref{eq-pa-ch}) yields (\ref{pa-ch}).
\end{proof}

\begin{lem}[Concatenation]
\label{concatenation}
For any $\bf a, \bf b\in \mathcal{W}^{\circ},\, \bf a\qar \bf b$, we have
\begin{equation}
C_{\Gamma}^{-1}\mu(D_{\bf a})\mu(D_{\bf b})\leq \mu(D_{\bf a' \bf b})\leq C_{\Gamma} \mu(D_{\bf a})\mu(D_{\bf b}).
\end{equation}
\end{lem}

\begin{proof}
This follows from Lemma \ref{lem:ia gammaa} similarly to Lemma \ref{lem-pa-ch}, using that $D_{\bf a' \bf b}=\gamma_{\bf a'}D_{\bf b}$.
\end{proof}

\begin{lem}[Reversal]\label{lem:reversal}
For any $\bf a\in \mathcal{W}^{\circ}$, we have
\begin{equation}
C_{\Gamma}^{-1} \mu(D_{\bf a})\leq \mu (D_{\overline{\bf a}})\leq C_{\Gamma}\mu(D_{\bf a}).
\end{equation}
\end{lem}

\begin{proof}
Without loss of generality, we may assume that $|\bf a|\geq 3$. We write $\bf a=a_1\ldots a_n$ and denote $\bf b:=a_2\ldots a_{n-1}$, so that $\bf a=a_1\bf b a_n$. Since $D_{\bf a}=\gamma_{a_1}(D_{\bf b a_n})$ and $D_{\bar{\bf a}}=\gamma_{\overline{a_n}}(D_{\overline{\bf b}\overline{a_1}})$, it suffices to show that
\begin{equation}
\label{conca}
C_{\Gamma}^{-1}\mu(D_{\bf b a_n})\leq \mu(D_{\ov{\bf b}\ov{a_1}}) \leq C_{\Gamma} \mu (D_{\bf b a_n}).
\end{equation}
By Lemma \ref{lem:lip derivative}, we have 
\begin{align*}
\mu(D_{\bf b a_n})\approx_{C_{\Gamma}} \|\gamma_{\bf b}\|_{\mathcal{S}}^{-2\delta} \mu(D_{a_n})\,\,\,\text{and}\,\,\, \mu(D_{\ov{\bf b} \ov{a_1}}) \approx_{C_{\Gamma}} \|\gamma_{\ov{\bf b}}\|_{\mathcal{S}}^{-2\delta}\mu(D_{a_1}).
\end{align*}
It follows from the definition of $\|\cdot\|_{\mathcal{S}}$ that $\|\gamma_{\bf b}\|_{\mathcal{S}}=\|\gamma_{\ov{\bf b}}\|_{\mathcal{S}}$. Hence Lemma \ref{lem:reversal} follows.
\end{proof}

\begin{lem}[Separation]
\label{lem-sep}
For any $\bf a\in \mathcal{W}^{\circ}$ and any $b,c\in \mathcal{A}$ so that $\bf a\to b$ and $\bf a \to c$, we have
\begin{equation}
\label{sep}
\dis_E(D_{\bf a b},D_{\bf a c})\geq \CG^{-1} \mu(D_{\bf a})^{1/\delta}.
\end{equation}
\end{lem}

\begin{proof}
Denote $\bf a=a_1\ldots a_n$. For any $x\in D_{\bf a b},\,y\in D_{\bf a c}$, set $\tilde{x}=\gamma_{\bf a'}^{-1}(x)\in D_{a_n b},\,\tilde{y}=\gamma_{\bf a'}^{-1}(y)\in D_{a_n c}$. Using Lemma  \ref{elementary} and \ref{lem:ia gammaa}, we obtain
\begin{equation}
|x-y|\geq \CG^{-1} |\tilde{x}-\tilde{y}| \mu(D_{\bf a})^{1/\delta}.
\end{equation}
As we have a uniform non-trivial lower bound for the Euclidean distance between the second generation of discs, (\ref{sep}) follows.
\end{proof}

\subsection{Patterson-Sullivan measures II}
\begin{lem}
\label{ps-est}
Let $\Omega$ be any Euclidean disc of radius $\sigma$ contained in $D_{a}$ for some $a\in \mathcal{A}$. Then 
\begin{equation}
\label{ps-est-eq}
\mu(\Omega)\leq C_{\Gamma} \sigma^{\delta}.
\end{equation}
\end{lem}

\begin{proof}
We may assume $\# (\Omega\cap \Lambda_{\Gamma})\geq 2$. Let $\bf a\in \mathcal{W}^{\circ}$ be the longest word such that $\Omega\cap \Lambda_{\Gamma}\subset D_{\bf a}$. Then there are two different $b,c\in \mathcal{A}$ so that $\bf a\to b,\,\bf a\to c$ and $\Omega\cap D_{\bf a b}\neq \emptyset,\,\Omega\cap D_{\bf a c}\neq \emptyset$. By Lemma \ref{lem-sep}, the distance between $D_{\bf a b}$ and $D_{\bf a c}$ is bounded from below by $C_{\Gamma}^{-1}\mu(D_{\bf a})^{1/\delta}$. Hence (\ref{ps-est-eq}) follows. 
\end{proof}

Armed with Section \ref{distortion-estimates}, the following two lemmas do follow directly from the arguments in \cite{BD1}.

\begin{lem}[Lemma 2.14 in \cite{BD1}]
Assume that $\tau\in (0,1],\,\bf b \in \mathcal{W}^{\circ}$. Then
\begin{equation}
\#\{\bf a\in \mathcal{W}^{\circ}| \,\bf b \prec \bf a, \,\mu (D_{\bf a})\geq \tau^{\delta}\}\leq C_{\Gamma} \tau^{-\delta}\mu (D_{\bf b}).
\end{equation}
\end{lem}

\begin{lem}[Lemma 2.15 in \cite{BD1}]
\label{tree-count}
Let $\Omega$ be any Euclidean disc of radius $\sigma$ contained in $D_{a}$ for some $a\in \mathcal{A}$. For all $C_0\geq 2$, we have
\begin{equation}
\#\{\bf a\in \mathcal{W}^{\circ}|\,\tau^{\delta} \leq \mu(D_{\bf a})\leq C_0 \tau^{\delta},\, D_{\bf a}\cap \Omega\neq \emptyset\}\leq C_{\Gamma}\tau^{-\delta} \sigma^{\delta}+C_{\Gamma}\log C_0.
\end{equation}
\end{lem}

\subsection{Partitions and transfer operators}
\label{pa-tran}
A partition $Z$ is a subset of words in $\mathcal{W}^\circ$ which is such that
$$\Lambda(\Gamma)=\bigsqcup_{\bf a\in Z} \left (\Lambda(\Gamma)\cap D_{\bf a}\right).$$
By the definition of Schottky groups, an obvious family of partitions is given for all $n\geq 1$ by 
$$Z=\mathcal{W}_n.$$
However, this natural choice turns out to be not the most convenient for our purpose, simply because elements corresponding to words with same
length may have very different distortion (derivative). Instead, similarly as in \cite{BD1}, we will consider $\tau>0$ a parameter (destined to be taken small later on), and set
\begin{equation}
\label{partition}
Z(\tau):=\{ \bf a \in \mathcal{W}^\circ\ |\ \mu(D_{\bf a})\leq \tau^\delta< \mu(D_{\bf a'})\}.
\end{equation}
The fact that for all $\tau>0$ small enough $Z(\tau)$ is a partition follows readily from Lemma \ref{lem:contraction} and its consequence: there exist uniform $C_\Gamma$ and $0<\rho<1$ such that for all $\bf a \in \mathcal{W}^\circ$, 
$$\mu(D_{\bf a})\leq C_\Gamma \rho^{\vert \bf a \vert}.$$
Notice that by definition of $Z(\tau)$ and using Lemma \ref{lem-pa-ch} we get that as $\tau \rightarrow 0$,
$$\tau^{-\delta}\leq \# Z(\tau)\leq C \tau^{-\delta}.$$
Moreover, we have the following estimate.
\begin{lem}
\label{lem:zctau}
	For $\tau>0,C>1$, let 
	\begin{equation}\label{equ:zctau}
	Z(C,\tau)=\{\bf b\in \mathcal{W}\ | \ C^{-1}\tau^{\delta}\leq \mu(D_{\bf b})\leq C\tau^{\delta} \}.
	\end{equation}
	Then there exists $l\in \bb N$ independent of $\tau$ such that
	\[Z(C,\tau)\subset Z(C\tau)\times \cup_{0\leq n\leq l}\cal W_n. \]
\end{lem}
\begin{proof}
For any $\bf b\in Z(C,\tau)$, by the construction of $Z(C,\tau)$, we can express $\bf b$ as $\bf b=\bf b' \bf b''$ with $\bf b'\in Z(C\tau)$. 
	Note that $|\bf b''|$ is bounded by a constant depending on $C$ due to the uniform contracting property (Lemma \ref{lem:contraction}). Then Lemma \ref{lem:zctau} follows. 
\end{proof}
To each partition $Z(\tau)$ we will associate a transfer operator $\lt_{Z(\tau)}$ acting on functions which is such that for all $f$ bounded Borel on $\Lambda(\Gamma)$, we have
$$\int_{\Lambda(\Gamma)} fd\mu= \int_{\Lambda(\Gamma)} \lt_{Z(\tau)}(f)d\mu.$$
Formula (\ref{equivariance}) shows that for all $j\in \mathcal{A}$,
$$\lt_{Z(\tau)}f(x)=\sum_{\bf a \in Z(\tau),\ \bf a \qar j} w_{\bf a'}(x) f(\gamma_{\bf a'}(x))\ \mathrm{if}\ x\in D_j.$$
This formula can be iterated to give
\begin{equation}
\label{transfer1}
\lt_{Z(\tau)}^kf(x)=\sum_{\bf a_1,\ldots, \bf a_k \atop \bf a_1\qar \ldots \qar \bf a_k\qar j} w_{\bf a_1'\ldots \bf a'_k}(x)f(\gamma_{\bf a_1'\ldots \bf a'_k}(x)).
\end{equation}
The rough strategy of the proof of Theorem \ref{main1} is then to write
$$\int_{\Lambda(\Gamma)} e^{it\varphi(x)}g(x)d\mu= \int_{\Lambda(\Gamma)} \lt_{Z(\tau)}^k(e^{it\varphi}g)d\mu,$$
one hopes to {\it catch cancellations} in the exponential sums
$$\sum_{\bf a_1,\ldots, \bf a_k \atop \bf a_1\qar \ldots \qar \bf a_k\qar j} w_{\bf a_1'\ldots \bf a'_k}(x)e^{it\varphi(\gamma_{\bf a_1'\ldots \bf a'_k}x)}g(\gamma_{\bf a_1'\ldots \bf a'_k}x),$$
with $\tau\asymp \vert t\vert^{-\beta}$, and $\beta>0$ suitably chosen. We will make this more precise in $\S$\ref{sp}.

 \section{Sum-products and decay of oscillatory integrals}
 \label{sp}
 
  In this section, we prove Theorem \ref{main1}. The key tool is an estimate on the decay of exponential sums (Proposition \ref{thm:sum product}). In Lemma \ref{lem:lip log'}, we've established the Lipschitz property of the derivatives of the elements in $\Gamma$. Using this and the H\"{o}lder inequality, we follow the scheme in \cite[Lemma 3.4, 3.5]{BD1} to obtain a combinatorial description of the oscillatory integral in concern which allows us to control it via certain exponential sums. We finish the proof of Theorem \ref{main1} by applying Proposition \ref{thm:sum product} to the exponential sum in (\ref{com-sum}).
  \begin{prop}\label{thm:sum product}
	Given $\kappa>0$, there exist $\epsilon>0$ and $k\in \mathbb{N}$ such that the following holds for $\eta\in\C $ with $|\eta|>1$. Let $C_0>0$ and let $\lambda_1,\cdots, \lambda_k$ be Borel measures supported on the annulus $\{z\in \mathbb{C}: 1/C_0\leq |z|\leq C_0\}$ with total mass less than $C_0$. Assume that each $\lambda_j$ satisfies the projective non concentration property, that is,
	\begin{equation}\label{equ:non concentration}
	\forall\sigma\in[C_0|\eta|^{-1},\,C_0^{-1}|\eta|^{-\epsilon}],\quad\sup_{a\in\bb R,\theta\in \R}\lambda_j\{z\in\C \ | \ |\Re(e^{i\theta}z)-a|\leq \sigma \}\leq C_0\sigma^{\kappa}.
	\end{equation}	
	Then there exists a constant $C_1$ depending only on $C_0,\kappa$ such that
	\begin{equation}
	\left|\int \exp(2\pi i\Re(\eta z_1\cdots z_k))\dd\lambda_1(z_1)\cdots \dd\lambda_k(z_k)\right|\leq C_1|\eta|^{-\epsilon}.
	\end{equation}
\end{prop}
As for the proof of the proposition, it has already been pointed out in \cite{BD1} that it can be shown by following the proof of Lemma 8.43 in \cite{bourgain2010discretized} and replacing the real version of the sum-product theorem \cite[Theorem 1]{bourgain2010discretized} by its complex version established in \cite[Proposition 2]{bourgain_spectral_2012}. We refer readers to \cite[Appendix 4.1]{li-sumproduct_2018} for more details.

\subsection{A combinatorial description of the oscillatory integral}
We now begin the proof of Theorem \ref{main1}. In this section $C$ is a  constant depending only on the Schottky data and the constants $M, M'$ in Theorem \ref{main1}. It may change from line to line. Let $k\in \mathbb{N}$ be the constant in Proposition \ref{thm:sum product}, which depends only on $\kappa$, which is fixed once for all and given by
Proposition \ref{thm:no concentration}. Let $t$ be the frequency parameter in (\ref{main equation}). Without loss of generality we may assume that $|t|\geq C$. Define the small number $\tau>0$ by
\begin{equation}
|t|=\tau^{-2k-3/2}.
\end{equation}
Let $Z(\tau)\subset \mathcal{W}^{\circ}$ be the partition defined in (\ref{partition}) and let $\mathcal{L}_{Z(\tau)}$ be the associated transfer operator, see $\S$\ref{pa-tran}. 

We follow the notation introduced in   \cite{BD1}:
\begin{itemize}
\item for $\bf a=a_1\ldots a_n\in \mathcal{W}^{\circ}$ and $z\in \mathbb{C}$, write $\bf a\wa z$ if $z\in D_{a_n}$; 

\item for $\gamma=\gamma_{\bf a}\in\Gamma$ with $\bf a=a_1\ldots a_n$, we write $\gamma\rightarrow z$ or $\bf a\rightarrow z$ if $z\notin D_{\ov{a_n}}$;

\item we denote
\begin{equation*}
\bf A=(\bf a_0,\ldots, \bf a_k)\in Z(\tau)^{k+1},\,\,\, \bf B=(\bf b_1,\ldots, \bf b_k)\in Z(\tau)^k;
\end{equation*}

\item we write $\bf A \leftrightarrow \bf B$ if and only if $\bf a_{j-1}\qar \bf b_j \qar \bf a_j$ for all $j=1,\ldots, k$;

\item if $\bf A \lar \bf B$, then we define the words $\bf A * \bf B:=\bf a_0' \bf b_1' \bf a_1' \bf b_2' \cdots \bf a_{k-1}' \bf b_{k}' \bf a_k'$ and $\bf A \# \bf B:=\bf a_0' \bf b_1' \bf a_1' \bf b_2'\cdots \bf a_{k-1}' \bf b_k'$;

\item denote by $b(\bf A)\in \mathcal{A}$ the last letter of $\bf a_k$;

\item for each $\bf a\in \mathcal{W}^{\circ}$, fix a point $x_{\bf a}\in D_{\bf a}$;

\item for $j\in \{1,\ldots, k\}$ and $\bf b\in Z(\tau)$ such that $\bf a_{j-1}\qar \bf b \qar \bf a_j$, define
\begin{equation*}
\zeta_{j,\bf A}(\bf b):=\tau^{-2} \gamma'_{\bf a_{j-1}'\bf b'}(x_{\bf a_j}).
\end{equation*}
\end{itemize}
Using the functions $\varphi,g$ from the statement of Theorem \ref{main1}, define
\begin{equation*}
f(x):=\exp (it \varphi(x)) g(x),\,\,\,x\in \Lambda_{\Gamma}.
\end{equation*}
By  (\ref{transfer1}), the integral in (\ref{main equation}) can be written as follows:
\begin{equation*}
\int_{\Lambda_{\Gamma}} fd\mu =\int_{\Lambda_{\Gamma}} \mathcal{L}_{Z(\tau)}^{2k+1} f d\mu=\sum_{\bf A, \bf B: \bf A\leftrightarrow \bf B} \int_{D_{b(\bf A)}} f(\gamma_{\bf A* \bf B}(x)) w_{\bf A * \bf B}(x) d\mu(x).
\end{equation*}
The following lemma follows almost the same lines with \cite[Lemma 3.4]{BD1}. This idea is to use the Lipschitz property of $w_{\bf A \# \bf B}$ (Lemma \ref{lem:lip log'}) to obtain an approximation for $w_{\bf A \# \bf B}(x)$ and then use Schwartz's inequality to get the following bound.
\begin{lem}\label{lem:os1}
We have
\begin{equation}
\label{os1}
\left\lvert \int_{\Lambda_{\Gamma}} f d\mu\right\lvert^2\leq C \tau^{(2k-1)\delta} \sum_{\bf A,\bf B: \bf A\leftrightarrow \bf B} \left\lvert  \int_{D_{b(\bf A)}} e^{it \varphi (\gamma_{\bf A *\bf B}(x))} w_{\bf{a}_k'}(x)d\mu(x)\right\lvert ^2+C\tau^2.
\end{equation}
\end{lem}

\begin{proof}
It follows from Lemma \ref{lem:ia gammaa} that for each $\bf a=a_1\ldots a_n\in Z(\tau)$, we have
\begin{equation}
\label{es1}
C^{-1}\tau^{\delta}\leq w_{\bf a'}(x)\leq C\tau^{\delta}\,\,\,\text{for}\,\,\, x\in D_{a_n}.
\end{equation}
This yields, using chain rule,
\begin{align}
\label{es2}
&C^{-1}\tau^{2k\delta}\leq w_{\bf A \# \bf B}(\gamma_{\bf a_k'}(x))\leq C\tau^{2k\delta},\\
\label{es3}
&C^{-1}\tau^{2k\delta}\leq w_{\bf A \# \bf B}(x_{\bf a_k})\leq C\tau^{2k\delta}.
\end{align}
Meanwhile, using Lemma \ref{lem:lip log'} and \ref{elementary}, we deduce that
\begin{equation}
\label{es}
\exp(-C\tau)\leq \left|\frac{w_{\bf A \# \bf B} (\gamma_{\bf a_k'}(x))}{w_{\bf A \# \bf B}(x_{\bf a_k})}\right|\leq \exp (C\tau).
\end{equation}
Observe that $|g(\gamma_{\bf A * \bf B}(x))-g(x_{\bf a_0})|\leq C \tau$. Combining this with (\ref{es1})-(\ref{es}), we obtain
\begin{equation*}
\left \lvert \int_{\Lambda_{\Gamma}} f d\mu -\sum_{\bf A,\bf B: \bf A\leftrightarrow \bf B} \int_{D_{b(\bf A)}} e^{i t \varphi (\gamma_{\bf A *\bf B}(x))} g(x_{\bf a_0})w_{\bf A \# \bf B}(x_{\bf a_k}) w_{\bf a_k'}(x)d\mu\right \lvert\leq C\tau.
\end{equation*}
Using Schwarz's inequality and (\ref{es3}), we get
\begin{align*}
&\left\lvert  \sum_{\bf A,\bf B: \bf A\leftrightarrow \bf B} \int_{D_{b(\bf A)}} e^{i t \varphi(\gamma_{\bf A *\bf B}(x))} g(x_{\bf a_0}) w_{\bf A \#\bf B}(x_{\bf a_k}) w_{\bf a_k'}(x)d\mu(x)\right\lvert^2\\
&\leq C\tau^{(2k-1)\delta} \sum_{\bf A,\bf B:\bf A\leftrightarrow \bf B} \left\lvert\int_{D_{b(\bf A)}} e^{i t\varphi (\gamma_{\bf A *\bf B}(x))}w_{\bf a_k'}(x)d\mu(x) \right\lvert^2,
\end{align*}
completing the proof of the lemma.
\end{proof}

To handle the first term on the right-hand side of (\ref{os1}), we estimate using (\ref{es1})
\begin{equation}\label{es5}
\begin{split}
&\sum_{\bf A,\bf B:\bf A\leftrightarrow \bf B}\left\lvert\int_{D_{b(\bf A)}} e^{i t\varphi (\gamma_{\bf A *\bf B}(x))}w_{\bf a_k'}(x)d\mu(x) \right\lvert^2\\
=&\sum_{\bf A} \int_{D_{b(\bf A)}^2} w_{\bf a_k'}(x) w_{\bf a_k'}(y)\sum_{\bf B:\bf A\leftrightarrow \bf B} e^{i t(\varphi(\gamma_{\bf A* \bf B}(x))-\varphi(\gamma_{\bf A *\bf B}(y)))}d\mu(x)d\mu(y)\\
\leq& C \tau^{2\delta} \sum_{\bf A} \int_{D_{b(\bf A)}^2} \left \lvert \sum_{\bf B:\bf A\leftrightarrow \bf B} e^{i t(\varphi(\gamma_{\bf A* \bf B}(x))-\varphi(\gamma_{\bf A *\bf B}(y)))} \right\lvert d\mu(x) d\mu(y).
\end{split}
\end{equation}

With Lemma \ref{lem:lip log'} available, the proof of the following lemma is almost the same as in \cite[Lemma 3.5]{BD1}.
\begin{lem}
\label{lem:os 2}
Denote
\begin{equation*}
J_{\tau}:=\{\eta\in \C|\, |\eta|\in[\tau^{-1/8},C\tau^{-1/2}] \}
\end{equation*}
where $C$ is sufficiently large. Then
\begin{equation}
\label{com-sum}
\left\lvert \int_{\Lambda_{\Gamma}} f d\mu  \right\lvert^2 \leq C\tau^{(2k+1)\delta} \sum_{\bf A} \sup_{\eta\in J_{\tau}} \left\lvert  \sum_{\bf B: \bf A\leftrightarrow \bf B} e^{2\pi i\Re(\eta \ze\cdots \zee )}   \right\lvert+C\tau^{\delta/4}.
\end{equation}
\end{lem}

\begin{proof}
Fix $\bf A$. Take $x,y\in D_{b(\bf A)}$ and put
\begin{equation*}
\tilde{x}:=\gamma_{\bf a_k'}(x),\,\,\, \tilde{y}:=\gamma_{\bf a_k'}(y)\in D_{\bf a_k},\,\,\,  v:=\frac{\tilde{x}-\tilde{y}}{|\tilde{x}-\tilde{y}|}.
\end{equation*}

Assume that $\bf A \leftrightarrow \bf B$. Note that $\varphi$ is real valued function defined on a neighborhood of $\Lambda_\Gamma$. For $z=x+iy$, we use the notation 
$$\varphi'(z)=\partial_x\varphi(z)-i\partial_y\varphi(z),$$ 
so that we can write
$$D_z\varphi (w)=\partial_x\varphi (z) w_1+\partial_y\varphi (z)w_2=\Re(\varphi'(z) w),$$ for $w=w_1+iw_2$. Since $\gamma_{\bf A *\bf B}(x)=\gamma_{\bf A \# \bf B}(\tilde{x})$, $\gamma_{\bf A * \bf B}(y)=\gamma_{\bf A \# \bf B}(\tilde{y})$, we 
 have
\begin{align*}
\varphi(\gamma_{\bf A * \bf B} (x))-\varphi (\gamma_{\bf A *\bf B}(y))=&\int_{0}^{|\tilde{x}-\tilde{y}|} (\varphi\circ \gamma_{\bf A \# \bf B}\circ p)'(s)ds \\
=&\int_{0}^{|\tilde{x}-\tilde{y}|} \Re \left(\varphi '(\gamma_{\bf A\# \bf B}\circ p(s)) \gamma_{\bf A \# \bf B}'(p(s)) v \right)ds,
\end{align*}
where $p:[0,\,|\tilde{x}-\tilde{y}|]\to \mathbb{C}$ is the path defined by $s\mapsto \tilde{y}+vs$.

Observe that, by Lemma  \ref{lem:lip log'} and \ref{elementary}, we have 
\begin{align*}
&|x_{\bf a_k}-p(s)|\leq \CG \tau\,\,\,\text{for any}\,\,\,s\in [0,|\tilde{x}-\tilde{y}|],\\
&|\gamma_{\bf a_{j}'\bf b_{j+1}'\ldots \bf b_{k}'}(p(s))-x_{\bf a_j}|\leq \CG \tau\,\,\,\text{for any}\,\,\,0\leq j\leq k-1\,\,\,\text{and}\,\,\,s\in [0,|\tilde{x}-\tilde{y}|]. 
\end{align*}
These yield for any $s\in [0,|\tilde{x}-\tilde{y}|]$
\begin{equation*}
|\varphi'(x_{\bf a_0})-\varphi'(\gamma_{\bf A\# \bf B}(p(s)))|\leq C \tau.
\end{equation*}
Hence we get
\begin{equation*}
\left\lvert  (\varphi \circ \gamma_{\bf A \# \bf B}\circ p)'(s)-\tau^{2k}\Re\left( \varphi'(x_{\bf a_0})\ze\cdots \zee v\right) \right\lvert\leq C \tau^{2k+1}.
\end{equation*}
It follows that
\begin{equation}\label{es6}
\left \lvert  \varphi(\gamma_{\bf A * \bf B}(y))-\varphi(\gamma_{\bf A *\bf B}(x))-\tau^{2k} \Re\left(\varphi'(x_{\bf a_0}) \ze\cdots \zee (\tilde{y}-\tilde{x})\right)\right\lvert\leq C \tau^{2k+2},
\end{equation}
where we recall that $\zeta_{j,\bf A}(\bf b):=\tau^{-2} \gamma'_{\bf a_{j-1}'\bf b'}(x_{\bf a_j})$.
Denote 
\[\eta:=\frac{\tau^{-3/2}}{2\pi}\varphi'(x_{\bf a_0})\cdot (\tilde{x}-\tilde{y})\cdot\operatorname{sign}(t). \]
Then by Lemma \ref{elementary} and \ref{lem:ia gammaa},
\[ C^{-1}M\tau^{-1/2}|x-y|\leq |\eta|\leq C\tau^{-1/2}|x-y|. \]
Notice that we have used here that on a neighborhood of $\Lambda_\Gamma$, 
$$M:=\inf_{z\in\mathcal{U}}\vert \varphi'(z) \vert =\inf_{z\in\mathcal{U}}\vert \nabla_z \varphi \vert>0.$$
Recall that $|t|=\tau^{-2k-3/2}$. By Lemma \ref{lem:os1}, \eqref{es5} and \eqref{es6}, we have
\[\left|\int_{\Lambda_\Gamma}f\dd\mu \right|^2\leq C\tau^{(2k+1)\delta} \sum_{\bf A} \int_{D_{b(\bf A)}^2} \left\lvert  \sum_{\bf B: \bf A\leftrightarrow \bf B} e^{2\pi i\Re(\eta \ze\cdots \zee) }   \right\lvert\dd\mu(x)\dd\mu(y)+C\tau^{1/2}. \]
By Lemma  \ref{ps-est},  for a fixed $C_0$
\[\mu\times\mu\{(x,y)\in\Lambda_\Gamma^2:|x-y|\leq C_0\tau^{1/4} \}\leq C\tau^{\delta/4}. \]
We therefore take the double integral over the set of $x,y$ such that $|x-y|\geq C_0\tau^{1/4}$, which assuming that 
$$M\geq \tau^{1/8}$$
implies for a large $C_0$ that $\eta\in J_{\tau}$. This finishes the proof.
\end{proof}

 \subsection{End of the proof of Theorem \ref{main1}}
 We will apply Proposition \ref{thm:sum product} to suitably defined discrete measures $\lambda_j$'s (see below) to estimate the sum in Lemma \ref{lem:os 2} and hence finish the proof of Theorem \ref{main1}.  The following technical proposition verifies that 
 these $\lambda_j$'s satisfy the required projective non-concentration property in  (\ref{equ:non concentration}).
 
 For any $\bf a, \bf b\in Z(\tau)$ and $x\in \mathbb{C}$ with $\bf a\wa \bf b \wa x$, write 
 \begin{equation*}
 \zeta_{\bf a,x}(\bf b):=\tau^{-2}\gamma'_{\bf a'\bf b'}(x).
  \end{equation*}
 \begin{prop}\label{thm:no concentration}
	 Assuming that $\Gamma$ is Zariski dense, there exist $C>0, \kappa>0$ with $\kappa$ depending only on the Patterson-Sullivan measure $\mu$ such that for any $\bf a\in Z(\tau)$, $x\in \mathbb{C}$ and $\sigma\in (\tau^{1/2},1)$, we have 
	 \[\sup_{a\in \mathbb{R},\,\theta\in \mathbb{R}}\tau^{\delta}\#\{\bf b\in Z(\tau)\ | \ \bf a\rightsquigarrow\bf b\rightsquigarrow x,\, |\Re(e^{i\theta}\zeta_{\bf a,x}(\bf b))-a|\leq \sigma \}\leq C \sigma^{\kappa} .\]
\end{prop}

Let us show how this proposition implies the main result. For each $\bf{A}\in Z(\tau)^{k+1}$ and for $1\leq j\leq k$, we define the following measure on $\mathbb{C}$:
\begin{equation}\label{equ:lambda j}
\lambda_j(E):=\tau^{\delta}\#\left\{\bf b\in Z(\tau)\,: \,\zeta_{j,\bf A}(\bf b)\in E \right\} \,\,\,\text{for any Borel set}\,\,\,E\subset \mathbb{C}.
\end{equation}
Notice that by Lemma \ref{lem:ia gammaa}, the chain rule, and the very definition of $Z(\tau)$, we know that the rescaled derivatives  $\zeta_{j,\bf A}(\bf b)$ satisfy uniformly
$$C_0^{-1}\leq \vert\zeta_{j,\bf A}(\bf b) \vert \leq C_0,$$
for some $C_0>0$, and $C_0$ can definitely be taken large enough so that the total mass of each $\lambda_j$ is less than $C_0$. Now recall that the constant $k$ in Proposition \ref{thm:sum product}, is determined by $\kappa$ from Proposition \ref{thm:no concentration}. Moreover we have:
\begin{itemize}
 \item $\vert t \vert=\tau^{-2k-3/2}$ and $\vert t \vert$ is taken large.
 \item $\vert \eta \vert \in [\tau^{-1/8}, C\tau^{-1/2}]$.
\end{itemize}
Therefore for each $\sigma \in [C_0\vert \eta \vert^{-1},C_0^{-1} \vert \eta\vert^{-\epsilon}]$, we get that 
$\sigma \in [C_0C^{-1}\tau^{1/2}, C_0^{-1} \tau^{\epsilon/8}].$ Taking again $C_0>0$ large enough so that $C_0C^{-1}> 1$, we can make sure that
$\sigma \in (\tau^{1/2},1)$ in order to apply Proposition \ref{thm:no concentration}. Hypothesis (\ref{equ:non concentration}) from Proposition \ref{thm:sum product} is now satisfied, we can combine it with Lemma \ref{lem:os 2} to obtain
$$\left\lvert \int_{\Lambda_{\Gamma}} e^{it\varphi}gd\mu  \right \lvert^2 \leq C\vert \eta \vert^{-\epsilon}+C\tau^{\delta/4}=O(\tau^{\epsilon/8})+O(\tau^{\delta/4})=O(\vert t \vert^{-\tilde{\epsilon}}),$$
with $\tilde{\epsilon}=\min\{\frac{\delta}{4(2k+3/2)}, \frac{\epsilon}{8(2k+3/2)} \}$. The proof is done.

\bigskip
We point out that this "non-concentration" result is really where the Zariski density hypothesis will be used and where our techniques deviate completely from the elementary arguments used in \cite{BD1}. Section $\S$\ref{non concentration} is fully devoted to the proof of this Proposition  \ref{thm:no concentration}.

 \section{Proving the non-concentration property}
  \label{non concentration}
  
   We prove Proposition \ref{thm:no concentration} in this section. Here is an overview of the strategy. Roughly speaking, we want to count the elements in $Z(\tau)$ whose derivatives lie in a neighborhood of a given affine line. We use the H\"{o}lder's inequality and reduce the problem to counting triples of elements whose derivatives are close to an affine line. A key observation is that the area of the triangle formed by such a triple must then be small (Lemma \ref{linear algebra}). The area (or determinant) condition enables us to obtain the desired supremum statement and hence Proposition \ref{lem:WNC} will lead to Proposition \ref{thm:no concentration}. In $\S$ \ref{real-rep}, we discuss real polynomials (defined by (\ref{real polynomial})) which are related to the determinant in  Proposition \ref{lem:WNC}. We establish an estimate regarding the measure of small values of a real polynomial (Lemma \ref{lem:strong PS}). Real proximal representations of $\operatorname{SL}_2(\mathbb{C})$ and Guivarc'h regularity property will naturally come into the picture. The last two subsections are about using Lemma \ref{lem:strong PS} to obtain Proposition \ref{lem:WNC}.  
   
 \subsection{Proof of Proposition \ref{thm:no concentration}}
 In the rest of the paper, given two real functions $f$ and $g$, we write $f\ll g$ if there exists a constant $C_1$ only depending on $C_\Gamma$ such that $f\leq C_1g$. We write $f\approx g$ if $f\ll g$ and $g\ll f$. Recall that $\mu$ is the Patterson-Sullivan measure of $\Gamma$ on the extended complex plane $\hat{\mathbb{C}}$. 
 
 Proposition \ref{thm:no concentration} follows from the following proposition. 
 \begin{prop}\label{lem:WNC}
	There exist $\epsilon=\epsilon(\mu)>0,\,N>0$ and $C>0$ such that for any $\bf a\in \cal W$, $\tau,\tau_1>0$, $1/N>\sigma>\tau,\tau_1$ and $z_o\in\C $ 
\begin{align}
\label{equ:PNC}
\# \big\{&(\bf b, \bf c, \bf d)\in Z(\tau)^3,\bf e\in Z(\tau_1)\ | \ \bf a \wa \begin{matrix} \bf b \\\bf c\\ \bf d
\end{matrix}\wa \bf e\wa z_o, \nonumber\\
& |\det(\gamma_{\bf a'\bf b'},\gamma_{\bf a'\bf c'},\gamma_{\bf a'\bf d'},\gamma_{\bf e'}z_o)|\leq \|\gamma_{\bf a}\|^{-4}\tau^2\sigma\big\}
\leq C\tau^{-3\delta}\tau_1^{-\delta} \sigma^{\epsilon},
\end{align}
where $\det(\gamma_1,\gamma_2,\gamma_3,z)$ is defined to be 
\begin{equation}
\label{det_mat}
	\det(\gamma_1,\gamma_2,\gamma_3,z)=\det\begin{pmatrix}
		\Re \gamma_1'z & \Im\gamma_1'z & 1\\
		\Re \gamma_2'z & \Im\gamma_2'z & 1\\
		\Re \gamma_3'z & \Im\gamma_3'z & 1
	\end{pmatrix}.
\end{equation}
\end{prop}
One of the advantages to consider determinant is that it yields an estimate regardless of the choice of affine lines as stated in Proposition \ref{thm:no concentration}. In \cite[Lemma 3.6]{BD1}, a weaker version of Proposition \ref{thm:no concentration} was proved.  

Note that the absolute value of (\ref{det_mat}) equals $\left\|u_1\wedge u_2 +u_2\wedge u_3+u_3\wedge u_1\right\|$ when taking $u_i=\gamma_i'z$, viewed as an element in $\R^2$. 
We need an elementary lemma in linear algebra. 
 
\begin{lem}[Corollary 3.6 in \cite{li2018thesis}]
\label{linear algebra}	
	Let $X_1,X_2, X_3$ be i.i.d. random vectors in $\R^2$ bounded by $C>0$. Let $l$ be an affine line in $\R^2$. Then for any $c>0$, we have
	\begin{equation}\label{equ:projective from strong}
	\bb P\{d(X_1,l)<c \}^{3}\leq \bb P\{\|X_1\wedge X_2+X_2\wedge X_3+X_3\wedge X_1\|<8Cc \}.
	\end{equation}
	%	and
	\begin{comment}\label{equ:strong from projective}
	\begin{split}
	&\bb P\{\|\sum(-1)^i X_1\wedge\cdots\wedge\hat{X}_i\wedge\cdots\wedge X_{d+1}\|<c \}\\
	&\leq\sum_{1\leq i\leq d}\bb P\{d(X_i,\sp_{\rm{aff}}(X_{d+1},X_1,\cdots,X_{i-1}))<c^{1/d} \}.
	\end{split}
	\end{comment}
\end{lem}
A similar statement can be found in \cite{eskin2005growth}. It says that if three points (in a bounded region) are near an affine line in $\R^2$, then the area of the triangle formed by these three points must be small, in a quantitative way. A key fact is that this bound for the area is {\it independent} of the affine line. 

We show how Proposition \ref{lem:WNC} yields Proposition \ref{thm:no concentration}. The proof is similar to the proof of Lemma 3.9 in \cite{li2018thesis}. We use H\"{o}lder inequality:  expressing $\bf b$ as a product of $\bf b_1$ and $\bf b_2$ we can apply H\"{o}lder inequality, which will allow us to use Lemma \ref{linear algebra}. 
\begin{proof}[From Proposition \ref{lem:WNC} to Proposition \ref{thm:no concentration}]
	Recall that $x\in D_b$ and $\bf a\in Z(\tau)$, we want to compute
	\begin{equation}\label{equ:bztau}
	\#\{\bf b\in Z(\tau)\ | \ \bf a\wa\bf b\wa b,\, |\Re(e^{i\theta}\gamma'_{\bf a'\bf b'}(x))-a|\leq \tau^2\sigma \}.
	\end{equation} 
	Let $\tau_1=\tau^{1/2}$. We divide $\bf b$ into $\bf b=\bf b_1'\bf b_2$ such that $\bf b_1\wa\bf b_2$ and $\bf b_1\in Z(\tau_1)$. Then $\bf b_2$ maybe not in $Z(\tau_1)$, but we have a control by Lemma \ref{concatenation}
	\begin{equation}\label{equ:b1b2}
	C_\Gamma^{-1}\mu(D_{\bf b_1})\mu(D_{\bf b_2}) \leq \mu(D_{\bf b}) \leq C_\Gamma \mu (D_{\bf b_1}) \mu(D_{\bf b_2}).
	\end{equation}
	Hence by \eqref{equ:b1b2} and \eqref{equ:zctau},
	\[Z(\tau)\subset Z(\tau_1)\times Z(C_\Gamma,\tau_1). \]
	Then \eqref{equ:bztau} is less than
	\begin{equation}\label{equ:b1zb2}	
		\#\{\bf b_1\in Z(\tau_1),\bf b_2\in Z(C_\Gamma,\tau_1)\ | \ \bf a\wa\bf b_1\wa \bf b_2\wa x,\,|\Re(e^{i\theta}\gamma'_{\bf a'\bf b_1'}(\gamma_{\bf b_2'}x)\gamma'_{\bf b_2'}(x))-a|\leq \tau^2\sigma \} .
	\end{equation}

	Fix $\bf b_2$, the length of $\gamma'_{\bf b_2'}(x)$ is approximately $\tau_1$, due to Lemma \ref{lem:ia gammaa}. We consider $\bf b_1$. 
	Let $X_1$ be the random variable in $\R^2$ given by $\gamma'_{\bf a'\bf b_1'}(\gamma_{\bf b_2'}x)$ (viewed a vector in $\R^2$) with $\bf b_1$ in $Z(\tau_1)$ and $\bf a\wa \bf b_1\wa \bf b_2$.  Let $$N_{\bf b_2}=\#\{\bf b_1\in Z(\tau_1)| \bf a\wa \bf b_1\wa \bf b_2\}.$$ Every choice of $\bf b_1$ has probability $1/N_{\bf b_2}$.Then \eqref{equ:b1zb2} equals
	\begin{equation}\label{equ:b2z}
	\begin{split}
	%	&\#\{\bf b_1\in Z(\tau),\bf b_2\in Z(C_\Gamma,\tau):\bf a\wa\bf b_1\wa \bf b_2\wa b| |\Re(e^{i\theta}\gamma'_{\bf a\bf b_1}(\gamma_{\bf b_2'}x)\gamma'_{\bf b_2}(x)-a|\leq \tau^2\sigma \}\\
	&\sum_{\bf b_2\in Z(C_\Gamma,\tau_1),\bf b_2\wa x}\#\{\bf b_1\in Z(\tau_1)\ | \ \bf a\wa\bf b_1\wa \bf b_2,\, |\Re(\gamma'_{\bf b_2'}(x)e^{i\theta}\gamma'_{\bf a'\bf b_1'}(\gamma_{\bf b_2'}x))-a|\leq \tau^2\sigma \}\\
	= &\sum_{\bf b_2\in Z(C_\Gamma,\tau_1),\bf b_2\wa x}N_{\bf b_2}\bb P_{\bf b_2}\left\{\left |\Re\left(\frac{\gamma'_{\bf b_2'}(x)}{|\gamma'_{\bf b_2'}(x)|}e^{i\theta}X_1\right)-a\right|\leq C_\Gamma\tau^{3/2}\sigma \right\}.
	\end{split}
	\end{equation}
	The definition of $X_1$ depends on $\bf b_2$ and we use $\bb P_{\bf b_2}$ to emphasize.
	Equation (\ref{equ:b2z}) means the distance of $X_1$ to a line $l$ is less than $C_\Gamma\tau^{3/2}\sigma$. Be careful that $X_1$ and the line $l$ depend on $\bf b_2$. The length of $X_1$ is approximately $\tau^{3/2}$. We write $ e^{i\theta'}:=\frac{\gamma'_{\bf b_2'}(x)}{|\gamma'_{\bf b_2'}(x)|}e^{i\theta}$.
	By \eqref{equ:projective from strong}, we have
	\begin{align*}
	\bb P_{\bf b_2}\{|\Re (e^{i\theta'}X_1)-a|\leq C_\Gamma\tau^{3/2}\sigma \}^3\leq \bb P_{\bf b_2}\{\|X_1\wedge X_2+X_2\wedge X_3+X_3\wedge X_1\|<C_\Gamma^2\tau^3\sigma \}.
	\end{align*}
	By the H\"older inequality, we have
	\begin{equation}\label{equ:1zb2}
	\begin{split}
	&\frac{1}{\#Z(C_\Gamma,\tau_1,x)}\sum_{\bf b_2}N_{\bf b_2}	\bb P_{\bf b_2}\{|\Re (e^{i\theta'}X_1)-a|\leq C_\Gamma\tau^{3/2}\sigma \}\\
	&\leq \left(	\frac{1}{\#Z(C_\Gamma,\tau_1,x)}\sum_{\bf b_2}	N_{\bf b_2}^3\bb P_{\bf b_2}\{\|X_1\wedge X_2+X_2\wedge X_3+X_3\wedge X_1\|<C_\Gamma^2\tau^3\sigma \}\right)^{1/3},
	\end{split}
	\end{equation}
	where $Z(C_\Gamma,\tau_1,x)=\{\bf b\in Z(C_\Gamma,\tau_1),\bf b\wa x\}$ and $\sum_{\bf b_2}=\sum_{\bf b_2\in Z(C_\Gamma,\tau_1,x)}$. Note that
	\begin{equation}\label{equ:b2x1x2}
	\begin{split}
	&\sum_{\bf b_2}	N_{\bf b_2}^3\bb P_{\bf b_2}\{\|X_1\wedge X_2+X_2\wedge X_3+X_3\wedge X_1\|<C_\Gamma^2\tau^3\sigma \}\\
	= & \# \{(\bf c, \bf d, \bf e)\in Z(\tau_1)^3,\bf b_2\in Z(C_\Gamma,\tau_1)\ | \ \\
	& \bf a \wa \begin{matrix} \bf c\\\bf d\\ \bf e
	\end{matrix}\wa \bf b_2\wa x, \, |\det(\gamma_{\bf a'\bf c'},\gamma_{\bf a'\bf d'},\gamma_{\bf a'\bf e'},\gamma_{\bf b_2'}x)|\leq C_\Gamma^2\tau^3\sigma\}.
	\end{split}
	\end{equation}
	Using Lemma \ref{lem:zctau}, we obtain
	\begin{equation}\label{equ:cde}
	\begin{split}
	&\# \{(\bf c, \bf d, \bf e)\in Z(\tau_1)^3,\bf b_2\in Z(C_\Gamma,\tau_1): \\
	&\bf a \wa \begin{matrix} \bf c\\\bf d\\ \bf e
	\end{matrix}\wa \bf b_2\wa x, \, |\det(\gamma_{\bf a'\bf c'},\gamma_{\bf a'\bf d'},\gamma_{\bf a'\bf e'},\gamma_{\bf b_2'}x)|\leq C_\Gamma^2\tau^3\sigma\}
	\\
	\leq &\#(\cup_{0\leq n\leq l}\cal W_n)\cdot \# \{(\bf c, \bf d, \bf e)\in Z(\tau_1)^3,\bf b_2\in Z(C_\Gamma\tau_1)\ | \ \\
	&\bf a \wa \begin{matrix} \bf c\\\bf d\\ \bf e
	\end{matrix}\wa \bf b_2\wa x', \, |\det(\gamma_{\bf a'\bf c'},\gamma_{\bf a'\bf d'},\gamma_{\bf a'\bf e'},\gamma_{\bf b_2'}x')|\leq C_\Gamma^2\tau^3\sigma\}.
	\end{split}
	\end{equation}
	Observe that $\tau^3\approx\|\gamma_{\bf a}\|^{-4}(\tau_1)^2$.
	Now we apply \eqref{equ:PNC} to estimate \eqref{equ:cde}. Then combining this with \eqref{equ:bztau}, \eqref{equ:b1zb2}, \eqref{equ:b2z}, \eqref{equ:1zb2} and \eqref{equ:b2x1x2}, we prove Proposition \ref{thm:no concentration}.
\end{proof}
 
 \subsection{Estimate on the measure of small values of a real polynomial}
\label{real-rep}
We introduce the following notion. Let $P$ be a polynomial in $z,\bar{z}$ with complex coefficients (not necessarily homogeneous). We call $P$ a {\it real polynomial} if 
\begin{equation}
\label{real polynomial}
P(z,\bar{z})=\overline{P(z,\bar{z})}.
\end{equation}
It is worthwhile to point out that the numerator of the determinant considered in Proposition \ref{lem:WNC} is a real polynomial. Recall that $\mu$ is the Patterson-Sullivan measure of $\Gamma$ on the extended complex plane $\hat{\C}$. We establish the following estimate.
\begin{lem}\label{lem:strong PS}
Fix $n>0$. There exist $C_n,\kappa_n>0$ with $\kappa_n$ depending only on the regularity of the push-forward measure $(e_{n})_*\mu$ on $\bb P V^*_{n}$ (defined in (\ref{induce-map})) such that the following holds. Let $P$ be a real polynomial in $z$ and $\bar z$ of highest degree $n$. Then for $0<r<1$ 
	\begin{equation}\label{equ:strong nuPS}
	\mu\{z\in \mathbb{C}\ | \ |P(z)|\leq r h(P) \}\leq C_nr^{\kappa_n},
	\end{equation}
where $h(P)$ is the maximum of the absolute values of the coefficients of $P$.	
Moreover,	for $0<\tau<r<1$ and $z\in\C$, we have
\begin{equation}
\label{equ:strong PS}
\#\left\{\bf d\in Z(\tau)\ | \ \bf d\qar z,\, \frac{|P(\gamma_{\bf d'}z)|}{h(P)}\leq r\right\}\leq C_n r^{\kappa_n}\# Z(\tau).
\end{equation}
	%\begin{equation}\label{equ:strong PS}
	%	\#\left\{\bf d\in Z(\tau), \bf d\wa z: \frac{|P(\gamma_{\bf d'}z)|}{h(P)}\leq r\right\}\leq C_n r^{\epsilon_n\#Z(\tau)
	%\end{equation}
\end{lem}

\subsubsection{Real proximal representations and the use of Guivarc'h regularity}
For the rest of this subsection, we regard the Schottky group $\Gamma$ in concern as a group in $\operatorname{SL}_2(\C)$. To prove Lemma \ref{lem:strong PS}, we consider real proximal irreducible representations of $\operatorname{SL}_2(\C)$.

Let $n$ be any nonnegative integer. Set $W_{n}$ to be the complex vector space of polynomials in $u_1,u_2,\overline{u}_1,\overline{u}_2$ that are homogeneous of degree $n$ in $(u_1,u_2)$ and homogeneous of degree $n$ in $(\overline{u}_1,\overline{u}_2)$. Define a representation $\tilde{\Phi}_{n}$ of $\operatorname{SL}_2(\mathbb{C})$ on $W_{n}$ by
\begin{equation}
\label{pro-rep}
\tilde{\Phi}_{n}\begin{pmatrix} a & b\\ c & d\end{pmatrix} P\begin{pmatrix} u_1\\ u_2\end{pmatrix}=P\left(\begin{pmatrix} a & b\\ c & d\end{pmatrix}^{-1} \begin{pmatrix} u_1\\ u_2\end{pmatrix}\right).
\end{equation}
This is a complex irreducible representation.

Let $V_{n}$ be the ``real" part of $W_{n}$. More precisely, we let $V_{n}$ be the real vector space consisting of polynomials $P(u_1,u_2)\in W_n$ satisfying
\begin{equation*}
\overline{P(u_1,u_2)}=P(u_1,u_2).
\end{equation*}
Note that elements in $V_n$ can be intuitively thought as homogeneous real polynomials. We have $V_{n}\otimes_{\mathbb{R}} \mathbb{C}\cong W_{n}$. Now define a representation $\Phi_{n}$ of $\operatorname{SL}_2(\mathbb{C})$ on $V_{n}$ as in (\ref{pro-rep}).  The induced representation by $\Phi_{n}$ on $V_{n}\otimes_{\mathbb{R}} \mathbb{C}$ is isomorphic to $\tilde{\Phi}_{n}$. So $\Phi_{n}$ is a real irreducible representation of $\operatorname{SL}_2(\mathbb{C})$.  As $\operatorname{SL}_2(\mathbb{C})$ is Zariski connected, $\Phi_n$ is strongly irreducible. Moreover, it is proximal. This is because $\Phi_{n}\begin{pmatrix} e^{t} & 0\\ 0 & e^{-t}\end{pmatrix}$ with $t>0$ is a proximal matrix: $\mathbb{R}u_2^n\bar{u}_2^n$ is the eigenspace that corresponds to the eigenvalue with the greatest absolute value.  

Let $V^{*}_{n}$ be the dual space of $V_{n}$. Denote the dual representation of $\operatorname{SL}_2(\mathbb{C})$ on $V^*_{n}$ by $\Phi^*_{n}$. It is also strongly irreducible and proximal. Consider the map 
\begin{equation*}
\tilde{e}_{n}:\mathbb{C}^2\to V^{*}_{n},\,\,\,(z_1,z_2)\mapsto \tilde{e}_n(z_1,z_2),
\end{equation*}where $\tilde{e}_{n}(z_1,z_2)$ is given by $\tilde{e}_n(z_1,z_2)\left(P(u_1,u_2)\right)=P(z_1,z_2)$ for any $P(u_1,u_2)\in V_n$. It induces the following map
\begin{equation} 
\label{induce-map}
e_n: \bb P^1_{\C}\to \bb P V^{*}_{n}
\end{equation}
which is $\operatorname{SL}_2(\mathbb{C})$-equivariant.

We fix an euclidean norm on $V^*_n$. In a finite dimensional vector space, different norms are equivalent. In particular, this norm is equivalent to the maximal norm. Note that the restriction of the maximal norm to the image of $\tilde{e}_n$ is equivalent to $|z_1|^{2n}+|z_2|^{2n}$. We take the operator norm on $V_n$.  For $x=\mathbb{R} v\in \mathbb{P}V_{n}$ and $y=\mathbb{R} h\in \mathbb{P} V^*_{n}$, define
\begin{equation*}
\Delta (x,y)=\frac{|h(v)|}{\|h\| \|v\|}.
\end{equation*}
	Recall that since the group $\Gamma$ is Zariski dense, we can use Guivarc'h's regularity property of $(e_n)_*\mu$ (see \cite[Theorem 14.1]{benoistquint}).
	\begin{lem}
	\label{bq-reg}
		There exist $C,\kappa_n>0$ such that for every $F\in \bb P V_{n}$, we have
		\[(e_n)_*\mu\left \{x\in\bb P V_{n}^*\big|\,\Delta (F, x)\leq r \right\}\leq Cr^{\kappa_n}. \]
	\end{lem}
	To be able to apply Theorem 14.1 from \cite{benoistquint}, we need to check that the Patterson-Sullivan measure $(e_n)^*\mu$ is a Furstenberg measure (i.e. a stationary measure) arising from a random walk on $\mathrm{GL}(V_n^*)$ with finite exponential moment. We also need to make sure that the Schottky group $\Gamma$ acts on $V^*_{n}$ proximally and strongly irreducibly. 
	%In fact, the condition in \cite[Theorem 14.1, Proposition 14.3]{benoistquint} for a $\mu$-stationary measure on is that $\Gamma_{\mu}$ is proximal, strongly irreducible and with finite exponential moment. 
	
	For the Patterson-Sullivan measure related to a convex co-compact group, it is exactly shown in the appendix  that it is a Furstenberg measure on $\hat{\C}$ with finite exponential moment arising from a random walk on $\Gamma$. Since the map \eqref{induce-map} is $\slc$-equivariant, the measure $(e_n)_*\mu$ is a Furstenberg measure on $\P V_n^*$ arising from a random walk on $\Phi_n^*(\Gamma)<\rm{GL}(V_n^*)$. For the proximal condition, we know that the subgroup $\Phi_n^*(\Gamma)$ is proximal iff its Zariski closure $\Phi_n^*(\slc)$ is proximal \cite{gol1989lyapunov}. For the strongly irreducible condition, it follows from the fact that the subgroup $\Phi_n^*(\Gamma)$ is strongly irreducible iff its Zariski closure is so \cite{gol1989lyapunov}. Furthermore, an irreducible representation of a connected group is always strongly irreducible. 

%Be careful that the polynomial $P$ here is not necessarily homogeneous. Recall that $d\wa z$ means $z\in I_{d_n}$.

%\begin{defi}
%	A real polynomial on $z,\bar z$ is a complex polynomial $P$ on $z,\bar z$ with $$P(z,\bar z)=\overline{P(z,\bar{z})}$$
%\end{defi}

\subsubsection{Proof of Lemma \ref{lem:strong PS}}
We show how to deduce Lemma \ref{lem:strong PS} from Lemma \ref{bq-reg}. 
\begin{proof}[Proof of Lemma \ref{lem:strong PS}]
The idea is to express the polynomial $P$ as a linear functional on $V_{n}^*$. Write 
\begin{equation*}
P(z,\bar{z})=\sum_{0\leq j+l\leq n}c_{j,l}\,z^{j}\bar{z}^l\,\,\,\text{with}\,\,c_{j,l}\text{'s complex numbers}.
\end{equation*}
 Set $f(u_1,u_2):=\sum_{0\leq j+l\leq n}c_{j,l} u_1^{j}u_2^{n-j}\bar{u}_1^{l}\bar{u}_2^{n-l}$. As $P$ is a real polynomial, we have $f\in V_{n}$. Take a homogeneous coordinate of $z\in\hat{\C}$, that is, $z=\frac{z_1}{z_2}$. 	Then we have
 	\[P(z)=f(\tilde{e}_n(z_1,z_2))/|z_2|^{2n}. \]
	Since the support of Patterson-Sullivan measure is bounded by $C$, we have
	\[|P(z)|=\frac{|f(\tilde{e}_n(z_1,z_2))|}{|z_2|^{2n}}\leq_C \frac{|f(\tilde{e}_n(z_1,z_2))|}{\|\tilde{e}_n(z_1,z_2)\|}.  \]
Note that because of finite dimension, $\|f\|$ is equivalent to $h(P)$. Therefore
	\begin{equation*}
	\frac{|P(z)|}{h(P)}\leq_C\frac{|f\left(\tilde{e}_n(z_1,z_2)\right)|}{\|f\|\|\tilde{e}_n(z_1,z_2)\|}=\Delta (\bb P f, e_n [z_1:z_2]). 
	\end{equation*}
	We apply Lemma \ref{bq-reg} to $(e_n)_*\mu$ and hence the first statement is proved. 
	
	Now we prove the second inequality \eqref{equ:strong PS}. The idea is to replace the counting by the Patterson-Sullivan measure. If $\bf d$ in $Z(\tau)$ satisfies the condition in \eqref{equ:strong PS}, then $\gamma_{\bf d'}z$ in $D_{\bf d}$. Using Lemma \ref{elementary}, we have for $w$ in $D_{\bf d}$
	\begin{align*}
		|P(w)|&\leq |P(\gamma_{\bf d'}z)|+|P(\gamma_{\bf d'}z)-P(w)|
		\leq h(P) r+h(P)C_\Gamma|w-\gamma_{\bf d'}z|\\
		&\leq h(P) r+h(P)C_\Gamma\tau\leq C_\Gamma h(P)r.
	\end{align*}
	Note that $\mu(D_{\bf d})\geq C_\Gamma^{-1}\tau^\delta$, we can replace the counting measure by $C_\Gamma\tau^{-\delta}\mu|_{D_{\bf d}}$ and replace the condition $|P(\gamma_{\bf d'}z)|\leq h(P)r$ by $|P(w)|\leq C_\Gamma h(P)r$.  Hence the left hand side of \eqref{equ:strong PS} is less than
	\begin{align*}
		C_\Gamma\tau^{-\delta}\mu\{w||P(w)|\leq C_\Gamma h(P) r \}.
	\end{align*}
	We then use \eqref{equ:strong nuPS} to finish the proof.
\end{proof}

\subsection{Proof of Proposition \ref{lem:WNC}: non-concentration of the real part}
%The starting point is the non concentration of points.
Let $C$ be a constant which depend only on $C_\Gamma$ and constants $C_n$ in Lemma \ref{lem:strong PS} and it may vary from line to line.

We prove the non concentration of the real part using Lemma \ref{lem:strong PS}. The main point is to verify that $h(P)$ is large. 
\begin{prop}
	\label{BD}
	There exists $\epsilon=\epsilon(\mu)>0$ such that for any $0<\tau,\tau_1\leq \sigma\leq 1/N_0$, $z_0\in \C$ and $\bf a\in \mathcal{W}^{\circ}$, we have
	\begin{align}
	\label{deri count}
	\#\{&(\bf b, \bf c)\in Z(\tau)^2,\bf d\in Z(\tau_1)\ | \ \bf a \wa\begin{matrix}
	\bf b\\ \bf c
	\end{matrix},\ \bf d\wa z_o,\nonumber\\
	&|\operatorname{Re}(\gamma'_{\bf a' \bf b'}(\gamma_{\bf d'}z_o)-\gamma'_{\bf a' \bf c'}(\gamma_{\bf d'}z_o))|\leq \|\gamma_{\bf a}\|_{\mathcal{S}}^{-2}\tau\sigma\}
	\leq C\tau^{-2\delta}\tau_1^{-\delta} \sigma^{\epsilon}.
	\end{align}
\end{prop}
%Please see \eqref{equ:bd domian} for $\wa$.
The idea is to prove that for ``most'' $\gamma_1,\gamma_2$ in $Z(C_\Gamma,\tau_2)$, where $\tau_2>0$ will be take as $\|\gamma_{\bf a}\|_{\mathcal{S}}^{-2}\tau$, we have
\begin{equation*}
\#\{\bf d\in Z(\tau_1)\ | \ \bf d\wa z_o,\, |\Re(\gamma_1'\gamma_{\bf d'}z_0)-\Re(\gamma_2'\gamma_{\bf d'}z_0)|\leq \tau_2\sigma \}\leq \sigma^{\epsilon}\#Z(\tau_1).
\end{equation*}

We make some computation
\begin{equation}\label{equ:regamma}
\begin{split}
	&\Re(\gamma_1'z-\gamma_2'z)=\Re\left(\frac{1}{(c_1z+d_1)^2}-\frac{1}{(c_2z+d_2)^2}\right)\\
	=&\Re\left(\frac{((c_2z+d_2)^2-(c_1z+d_1)^2)\overline{(c_1z+d_1)^2}\overline{(c_2z+d_2)^2}}{|c_1z+d_1|^4|c_2z+d_2|^4}\right).
\end{split}
\end{equation}
Let $P(z)$ be the real part of the numerator. We need to estimate $h(P)$. The following lemma is about the pairs $(\bf b, \bf c)$'s which yield the polynomials $P$ that might have small $h(P)$. %\textcolor{red}{A natural idea is to compute the measure that $c_1^2=c_2^2$, but this is like a local limit theorem for the coefficients. In the non proximal case, the limit laws for coefficient is difficult. See Benoist-Quint Chapter 14} Here we want to use the difference of $\gamma_2^{-1}\infty-\gamma_1^{-1}\infty$, which is available. 
\begin{lem}
	\label{inf count}
	Let $\sigma\geq\tau$. For each $\bf b\in Z(\tau)$ such that $\bf a \wa \bf b$, we have
	\begin{equation}
	\label{eq-inf-count}
	\#\{\bf c\in Z(\tau)\ | \ \bf a \wa \bf c, |\gamma^{-1}_{\bf a' \bf b'}(\infty)-\gamma^{-1}_{\bf a' \bf c'}(\infty)|\leq \sigma\}\leq C \tau^{-\delta} \sigma^{\delta}.
	\end{equation}
\end{lem}

\begin{proof}
	Denote $\bf e=\overline{\bf c'}$. We have $\gamma^{-1}_{\bf a' \bf c'}(\infty)=\gamma_{\bf e \overline{\bf a'}}(\infty)\in D_{\bf e}$. Also, $C_{\Gamma}^{-1}\tau^{\delta} < \mu (D_{\bf e})<C_{\Gamma}^2 \tau^{\delta}$ by Lemma \ref{lem-pa-ch} and \ref{lem:reversal}. Therefore, the left-hand side of (\ref{eq-inf-count}) is bounded by
	\begin{equation*}
	2r\cdot \#\{\bf e\in \mathcal{W}^{\circ}\ | \ C_{\Gamma}^{-1}\tau^{\delta}<\mu (D_{\bf e})<C_{\Gamma}^2\tau^{\delta},\,D_{\bf e}\cap \Omega\neq \emptyset\},
	\end{equation*}
	where $\Omega$ is the disk of radius $\sigma$, centered at $\gamma^{-1}_{\bf a' \bf b'}(\infty)$. Now (\ref{eq-inf-count}) follows from Lemma \ref{tree-count}.
\end{proof}

\begin{lem}\label{lem:non real}
	Let $\tau_1,\tau_2>0$ and $\sigma>\tau_1,\tau_2$. Let $\gamma_1,\gamma_2\in Z(C_\Gamma,\tau_2)$ and $A_1=\gamma_1^{-1}\infty-\gamma_2^{-1}\infty$. If $|A_1|\geq \sigma^{1/12}$, then we have
	\begin{equation}\label{equ:non real}
	\#\{\bf d\in Z(\tau_1)\ | \  \bf d\wa z_o,\ |\Re(\gamma_1'\gamma_{\bf d'}z_o)-\Re(\gamma_2'\gamma_{\bf d'}z_o)|\leq \tau_2\sigma \}\leq C\sigma^{\epsilon}\#Z(\tau_1),
	\end{equation}
	where $\epsilon=\kappa_6/2$ in Lemma \ref{lem:strong PS}. 
\end{lem}
\begin{proof}
%	We distinguish two case, the first case is that $|\gamma_{\bf d'}z_o-\gamma_i^{-1}\infty|<\sigma^{1/16}$ for $i=1$ or $2$. By Lemma \ref{tree-count}, the number of such of $\bf d$'s is less than $C_\Gamma\tau^{-\delta}\sigma^{\delta/16}$.
	For $i=1,2$, we have $\|\gamma_i\|_{\mathcal{S}}\approx \tau_2^{-1/2}$ by Lemma \ref{lem:lip derivative} and Lemma \ref{lem:ia gammaa}.
		Observe that for $z\in\bf D$, the union of disks $D_j$, we have $|c_iz+d_i|\leq C_\Gamma \|\gamma_i\|_{\mathcal{S}}\leq C_\Gamma^2 \tau_2^{-1/2}$. This implies the denominator of \eqref{equ:regamma} is less than $C_\Gamma^{16}\tau_2^{-4}$ for $z\in\bf D$. Now $\gamma_{\bf d'}z_o$ is in $\bf D$ for $\bf d\in Z(\tau_1), \bf d\wa z_o$. Hence by the formula \eqref{equ:regamma} of $\Re(\gamma_1'z-\gamma_2'z)$, it is sufficient to prove that
		\begin{equation}
		\#\{\bf d\in Z(\tau_1)\ | \ \bf d\wa z,  |P(\gamma_{\bf d'}z_o)|\leq C_\Gamma^{16} \tau_2^{-3}\sigma \}\leq C\sigma^{\epsilon}\#Z(\tau_1).
		\end{equation}
		It suffices to prove that $h(P)$ is greater than $c_1\tau_2^{-3}\sigma^{1/2}$, where $c_1>0$ is a constant only depends on $C_\Gamma$, because then we can apply \eqref{equ:strong PS} to $P$ with $r=\sigma^{1/2}C_\Gamma^{16}/c_1$. 
	
	In order to prove $h(P)$ large, we will prove that for some choice of $z$ with bounded norm, the value $P(z)$ is large. WLOG, suppose that $|c_2|\geq |c_1|$. Take $$z=A-\frac{d_1}{c_1},$$
	 where $A$ will be determined later. Then
	\begin{align*}
		P(z)=\Re\left((c_2^2(A_1+A)^2-c_1^2A^2)\overline{c_1^2c_2^2(A+A_1)^2A^2}\right).
	\end{align*}
	We will take $|A|$ small to get rid of the minus and $|A|$ not too small to have a lower bound. Take $|A|=|A_1|/10$, then the angle of the above formula almost only depends on $A$. With a suitable choice of the angle of $A$, the value of $P(z)$ is almost the absolute value. As $|c_i|=\|\gamma_i\|_{\mathcal{S}}\approx \tau_2^{-1/2}$, we obtain
	\begin{equation*}
		P(z)\gg |c_2^2A_1^2||c_1^2c_2^2A_1^4|\gg \tau_2^{-3}|A_1|^6\geq\tau_2^{-3}\sigma^{1/2}.
	\end{equation*}
	Since the norm of $z$ is bounded by $C_\Gamma$, we see that $h(P)\gg \tau_2^{-3}\sigma^{1/2}$. The proof is complete.
\end{proof}

Combining Lemma \ref{lem:non real} and Lemma \ref{inf count}, we obtain Proposition \ref{BD}.

\subsection{Proof of Proposition \ref{lem:WNC}: uniform non-concentration}
We complete the proof of Proposition \ref{lem:WNC} in this subsection. The idea is the same with Proposition \ref{BD}: we find conditions on $\gamma_1,\gamma_2,\gamma_3$ such that the real polynomial $P(z)$ showing up in the determinant has reasonable large height $h(P)$. 

The following lemma is similar to Lemma \ref{lem:non real}, but much more involved. 
\begin{lem}\label{lem:non proj}
	Let $\tau_1,\tau_2>0$, $\sigma>\tau_1,\tau_2$ and $z_o\in\C $. Let $\gamma_i\in Z(C_\Gamma,\tau_2)$, $i=1,2,3$. If $A_1=\gamma_3^{-1}\infty-\gamma_1^{-1}\infty,A_2=\gamma_3^{-1}\infty-\gamma_2^{-1}\infty$ satisfy 
	\begin{equation}\label{equ:A1A2}
	|A_1|,\,|A_2|\geq \sigma^{1/128}
	\end{equation} 
	and for $z_3=\gamma_3^{-1}\infty$,
	\begin{equation}\label{equ:real}
			|\Re(\gamma_1'z_{3})-\Re(\gamma_2'z_{3})|\geq \tau_2\sigma^{1/8}, 
	\end{equation}
	 then 
	\begin{equation}\label{equ:non proj}
	\#\{\bf e\in Z(\tau_1)\ | \ \gamma_i\rightarrow \bf e\wa z_o,\, |\det(\gamma_1,\gamma_2,\gamma_3,\gamma_{\bf e'}z_o)|\leq \tau_2^2\sigma \}\leq C\sigma^{\epsilon}\#Z(\tau_1),
	\end{equation}
	where $\epsilon=\kappa_8/4$ in Lemma \ref{lem:strong PS}.
\end{lem}

We first show how the above lemma will lead to Proposition \ref{lem:WNC}.
\begin{proof}[Proof of Proposition \ref{lem:WNC}]
	Let $\gamma_1=\gamma_{\bf a'\bf b'},\gamma_2=\gamma_{\bf a'\bf c'},\gamma_3=\gamma_{\bf a'\bf d'}$. By Lemma \ref{concatenation}, they are in $Z(C_\Gamma,\tau_2)$ with $\tau_2=\|\gamma_{\bf a}\|^{-2}\tau$.
	
	We have a dichotomy. If $\gamma_1,\gamma_2,\gamma_3$ satisfy the conditions in Lemma \ref{lem:non proj}, then the number of $\bf e$ is small, less than $\#Z(\tau_1)\sigma^\epsilon$.
	
	If not, the condition on $A_1,A_2$ can be dealt with Lemma \ref{inf count}. That is the number of $\bf b,\bf c,\bf d$ not satisfying \eqref{equ:A1A2} is small by Lemma \ref{inf count}.
	
	The main difficulty is to verify \eqref{equ:real}, but this can be dealt with Proposition \ref{BD}. Because for $\gamma_3=\gamma_{\bf a'\bf d'}$, then $$z_3=\gamma_{\bf a'\bf d'}^{-1}\infty=\gamma_{\bf{\bar d'}\bf{\bar{a}'}}\infty=\gamma_{(\bf{\bar d'})'}\gamma_{\bf{\bar{a}}}\infty,$$ where $\overline{\bf d'}\,\overline{\bf{a}'}=\overline{d_{n-1}}\cdots \overline{d_1}\overline{a_{m-1}}\cdots \bar{a_1}=\overline{d_{n-1}}\cdots \bar{d_2}\overline{a_{m}}\cdots \bar{a_1}=(\overline{\bf d'})'{\bf{\bar{a}}}$ and we have $\bf{\bar{d'}}\wa\bf{\bar{a}}$.
	Let $\bf f=\bar{\bf d'}$. The element $\bf f$ is not always in $Z(\tau)$. By Lemma \ref{lem:reversal}, it is in $Z(C_\Gamma,\tau)$. The number of $\bf b,\bf c,\bf d$ not satisfying \eqref{equ:real} is less than
	\[\#\{(\bf b,\bf c)\in Z(\tau),\bf f\in Z(C_\Gamma,\tau)\ | \ \bf a\wa \begin{matrix}\bf b \\ \bf c\end{matrix},\,\bf f\wa z,\,|\Re(\gamma_{\bf a'\bf b'}'(\gamma_{\bf f'}z))-\Re(\gamma_{\bf a'\bf c'}'(\gamma_{\bf f'}z))|\leq \tau_2\sigma^{1/8} \},\]
	where $z=\gamma_{\bar{\bf a}}\infty$ and $\tau_2=\|\gamma_{\bf a}\|^{-2}\tau$.
	Then by Lemma \ref{lem:zctau} and Proposition \ref{BD}, the proof is complete.
\end{proof}
It remains to prove Lemma \ref{lem:non proj}.
\begin{proof}[Proof of Lemma \ref{lem:non proj}]
	By the same argument as in the proof of Lemma \ref{lem:non real}, we have an upper bound of the denominator of $\det(\gamma_1,\gamma_2,\gamma_3,z)$ with $z\in \bf D$, which is less than $C_\Gamma^{24}\tau_2^{-6}$. Therefore it is enough to prove that
		\begin{equation*}
		\#\{\bf e\in Z(\tau_1)\ | \ \bf e\wa z_o, |P(\gamma_{\bf e'}z_o)|\leq C_\Gamma^{24}\tau_2^{-4}\sigma \}\leq C\sigma^{\epsilon}\#Z(\tau_1),
		\end{equation*}
		where the polynomial $P(z)$ is the numerator of $\det(\gamma_1,\gamma_2,\gamma_3,z)$, given by
		\[P(z)=\det\begin{pmatrix}
		\Re \overline{(c_1z+d_1)^2} & \Im \overline{(c_1z+d_1)^2} & |c_1z+d_1|^4\\
		\Re \overline{(c_2z+d_2)^2}& \Im \overline{(c_2z+d_2)^2} & |c_2z+d_2|^4\\
		\Re \overline{(c_3z+d_3)^2} & \Im \overline{(c_3z+d_3)^2}  & |c_3z+d_3|^4
		\end{pmatrix}. \]
		It suffices to prove that $h(P)$ is greater than $c_1\tau_2^{-4}\sigma^{3/4}$, where $c_1>0$ is a constant depending on $C_\Gamma$, because then we can apply Lemma \ref{lem:strong PS} \eqref{equ:strong PS} to $P$ with $r=\sigma^{1/4}C_\Gamma^{24}/c_1$. 
		
		In order to prove $h(P)$ large, we will prove that for some choice of $z$ with bounded norm, the value $P(z)$ is large. We take $z=A-\frac{d_3}{c_3}=A+z_3$, where $A$ will be determined later. Then
			\begin{equation}\label{equ:pzmatrix}			
			P(z)=-\det\begin{pmatrix}
			\Re  {\left(c_1(A_1+A)\right)^2} & \Im  {\left(c_1(A_1+A)\right)^2} & |c_1(A_1+A)|^4\\
			\Re  {\left(c_2(A_2+A)\right)^2}& \Im  {\left(c_2(A_2+A)\right)^2} & |c_2(A_2+A)|^4\\
			\Re  {(c_3A)^2} & \Im  {(c_3A)^2}  & |c_3A|^4
			\end{pmatrix}.
			\end{equation}
			We first fix the angle of $A$ such that $(c_3A)^2$ is an imaginary number, that is
			\begin{equation}\label{equ:reala}
			\Re {(c_3A)^2}=0.
			\end{equation}
			We let $|A|=\sigma^{1/4}$ and we claim that $|P(z)|\gg \tau_2^{-4}\sigma^{3/4}$.
			
			Now, we expand the determinant \eqref{equ:pzmatrix} with respect to the last line, using \eqref{equ:reala}, which gives
			\begin{equation}\label{equ:pz}
			P(z)=P_1+P_2,
			\end{equation}
			with
			\[P_1:=\Im (c_3A)^2\,\Re\left(\gamma_1'z-\gamma_2'z\right)\,|\gamma_1'z|^{-2}|\gamma_2'z|^{-2} \]
			and 
			\[ P_2:=|c_3A|^4\,\left(\Re  {\left(c_1(A_1+A)\right)^2}\,\Im  {\left(c_2(A_2+A)\right)^2} -	\Re  {\left(c_2(A_2+A)\right)^2}\,\Im  {\left(c_1(A_1+A)\right)^2}\right).\]
			Due to $\gamma_i\in Z(C_\Gamma,\tau_2)$, by Lemma \ref{lem:norm c} we know that $|c_1|,\,|c_2|\ll \tau_2^{-1/2}$. By $|A_1+A|,\,|A_2+A|\ll 1$, we obtain 
			\begin{equation}\label{equ:p2}
			|P_2|\ll |c_3A|^4\tau_2^{-2}.
			\end{equation}
			Let 
			\[ B(z)=\Re(\gamma_1'z-\gamma_2'z)|\gamma_1'z|^{-2}|\gamma_2'z|^{-2}, \]
			then $P_1=\Im(c_3A)^2B(z)$.
			The coefficients of $B$ are monomials of degree 6 on $c_1,d_1,c_2,d_2$. By Lemma \ref{lem:norm c}, we obtain  
			$$h(B)\leq\sup\{\|\gamma_1\|_{\mathcal{S}},\|\gamma_2\|_{\mathcal{S}} \}^6\ll C\tau_2^{-3}.$$
			 Due to $|A|=\sigma^{1/4}$, we know that
			\begin{equation}\label{equ:bzbz2}
			|B(z)-B(z_3)|\ll |z-z_3|h(B)\ll |z-z_3|\tau_2^{-3}= |A|\tau_2^{-3}= \sigma^{1/4}\tau_2^{-3}.
			\end{equation}
			Thanks to $|A_1|,|A_2|\geq \sigma^{1/128}$ \eqref{equ:A1A2}, we obtain $|\gamma_j'z_3|^{-1}\geq |c_j|^2|A_j|^2\gg \tau_2^{-1}\sigma^{2/128}$ for $j=1,2$. Combining with \eqref{equ:bzbz2} and \eqref{equ:real}, we obtain
			\begin{equation}\label{equ:bz}
			|B(z)|\geq |B(z_3)|-|B(z)-B(z_3)|\geq\sigma^\alpha\tau_2\sigma^{1/8}\tau_2^{-4}\sigma^{8/128}-\sigma^{-\alpha}\sigma^{1/4}\tau_2^{-3}\gg \tau_2^{-3}\sigma^{1/4}.
			\end{equation} 
			(Here we take $\alpha=1/64$. Due to $\sigma\leq 1/N$, we can take $N$ large enough such that $N^{\alpha}$ is greater than the constants that appeared, depending on  $C_\Gamma$.)
			Hence by \eqref{equ:pz}, \eqref{equ:p2} and \eqref{equ:bz}, we conclude that
			\begin{align*}
			|P(z)|&\geq |P_1|-|P_2|\geq |c_3A|^2|B(z)|-|c_3A|^4\tau_2^{-2}\\
			&\gg \tau_2^{-1}\sigma^{2/4}\tau_2^{-3}\sigma^{1/4}-\sigma^{-\alpha}\tau_2^{-2}\sigma^{4/4}\tau_2^{-2}=\tau_2^{-4}(\sigma^{3/4}-\sigma^{1-\alpha})\gg\tau_2^{-4}\sigma^{3/4}.
			\end{align*}
			This is what we need, then the bound $|z|\leq C_\Gamma$ implies $h(P)\gg \tau_2^{-4}\sigma^{3/4}$.
\end{proof}

\appendix
\gdef\thesection{Appendix A.}
	
\section{Exponential moment and Stationarity of Patterson-Sullivan measures. }
\begin{center} 
	Jialun LI
\end{center}

\gdef\thesection{\Alph{section}}

\def\slr{\rm{SL}(2,\mathbb R)}
\def\son{G}
\def\B{\bb B }
\def\S{\bb S}
\def\R{\bb R}
\def\P{\bb P}
\def\N{\bb N}
\def\H{\bb H}
\def\xmg{x^m_{\gamma}}
\def\limit{\Lambda(\Gamma)}

%\subsection{}\label{sec:exponential-moment}
%need an argument about the support? it is possible but tedious

In this appendix, we give a construction of a random walk on a convex cocompact subgroup of $\rm{SO}_0(1,n)$, which has exponential moment and such that the associated Patterson-Sullivan measure of the convex cocompact group can be realized as a stationary measure.

Let $\Gamma$ be a convex cocompact subgroup of $\son=\rm{SO}_0(1,n)$ ($n\geq 2$) and $\mu$ be an associated Patterson-Sullivan measure on the boundary at infinity $\partial \H^n$. Let $\nu$ be a Borel probability measure on $G$. We call $\mu$ a $\nu$-stationary measure or a Furstenberg measure if 
$$\mu=\nu*\mu:=\int_G\gamma_*\mu\,\dd\nu(g).$$
In this appendix, we provide a construction of a measure $\nu$ on $\Gamma$ such that $\mu$ is $\nu$-stationary and $\nu$ has a finite exponential moment, that is there exists $\epsilon_e>0$ such that $\int_G\|\gamma\|^{\epsilon_e}\dd\nu(\gamma)<\infty.$

For a measure $\nu$ on $\son$, we let $\Gamma_\nu$ be the subgroup generated by the support of $\nu$.
\begin{thm}\label{thm:expmom}
	Let $\Gamma$ be a convex cocompact subgroup of $\son$, and let $\mu$ be the Patterson-Sullivan measure on the limit set $\Lambda_{\Gamma}$. Then there exists a probability measure $\nu$ on $\Gamma$ with a finite exponential moment such that $\mu$ is $\nu$-stationary and $\Gamma_\nu=\Gamma$. 
\end{thm}
\begin{rem}
	\begin{enumerate}
	\item In  \cite{lalley1986gibbs} and \cite{lalley1989renewal}, Lalley announced the existence of such a $\nu$ for Schottky groups. But Lalley's proof only works for Schottky semigroups. In \cite{connell2007harmonicity}, the authors proved the existence of such a $\nu$ without the moment condition in the geometrically finite case. Our construction combines the methods of Connell-Muchnik and Lalley.
	
	\item For cocompact lattices, the construction is due to Furstenberg \cite{furstenberg1963poisson}. When the Hausdorff dimension $\delta_{\Gamma}\geq(n-1)/2$, the construction is due to Sullivan \cite{Sullivan2}. Their methods are based on the discretization of Brownian motions on hyperbolic spaces.
	
	\item On the other hand, for the geometrically finite with cusps case, it is impossible to find such a measure $\nu$ with exponential moment, because the finite exponential moment condition is impossible for noncompact lattice $\Gamma$ in $\slr$. It is shown that if $\nu$ is a measure on $\Gamma$ with a finite first moment, then the $\nu$-stationary measure $\mu$ is singular with respect to the Lebesgue measure (This fact is due to Guivarc'h and Le Jan \cite{guivarch1993winding}. See also \cite{deroin2009circle} and \cite{blachere2011harmonic}).
%	\item For Schottky subgroups, if we ask for finite support of $\nu$, then it is not possible. Lessa Pablo told me that he has a proof. But if we ask for finite support of $\nu$ for cocompact lattices, it is still an open question. 
	
	\item The result for $\rm{SO}_0(1,2)$ has already been announced and used in \cite{Li1}.
	\end{enumerate}
\end{rem}

\subsection{Basic properties and cover}
We will use the ball model for the hyperbolic $n$-space and fix the origin point $o$ in $X=\B^n=\{x\in\R^n| \|x\|^2=x_1^2+\cdots x_n^2<1 \}$. The hyperbolic riemannian metric $d$ at $x$ in $X$ is given by 
\[\frac{4\dd x^2}{(1-\|x\|^2)^2}. \]
The infinity $\partial X$ is isomorphic to the sphere $\S^{n-1}$. For $x,y$ in $X$ and $C>0$, let $\cal O_{C}(x,y)$ be the shadow of a ball centred at $y$ of radius $C$ seen from $x$, that is, the set of $\xi$ in the boundary $\partial X$ such that the geodesic ray issued from $x$ with limit point $\xi$ intersects the ball $B(y,C)$.

Let $hull(\Lambda_{\Gamma})$ be the convex hull of the limit set $\Lambda_{\Gamma}$ in $X\cup\partial X$. Without loss of generality, we suppose that $o$ is in the convex hull.
Since the group $\Gamma$ is convex cocompact, the quotient $C(\Gamma)=\Gamma\backslash (hull(\Lambda_{\Gamma})\cap X)$ is compact. Let 
\begin{equation}\label{equ:c0}
C_0=6\max\{\text{the diameter of the quotient set }C(\Gamma),2 C_1 ,2, \log C_2\},
\end{equation}
where $C_1, C_2$ are defined in Lemma \ref{lem:shadow} and Lemma \ref{lem:fg} respectively.

For an element $\gamma\in G$, we write $\xmg$ for the intersection of the ray $o,\gamma^{-1}o$ with the boundary $\partial X$. For $\gamma\in\Gamma$, let $\kappa(\gamma)=d(o,\gamma o)$ and let $r_{\gamma}=e^{-\kappa(\gamma)}$. Set $B_\gamma=\cal O_{C_0}(o,\gamma^{-1}o)$.   %it is equal to $B(\xmg,C_1r_{\gamma})$ the ball in $\partial X$ centred at $\xmg$ with radius $C_1r_{\gamma}$.

Recall the Busemann function and the Patterson-Sullivan measure. Recall \eqref{equ:quasi-invariant}, that is for $\gamma$ in $\Gamma$ and $\xi\in\partial X$, we have
\begin{align}
\frac{\dd \gamma_*\mu}{\dd\mu}(\xi)=e^{-\delta B_\xi(\gamma^{-1}o,o)}.
\end{align}
Let $f_ \gamma(\xi)=e^{-\delta B_\xi(\gamma^{-1}o,o)}$. For the stationary equation $\sum_{\gamma\in\Gamma}\nu(\gamma)\gamma_*\mu=\mu$, it is sufficient to verify
\begin{equation}\label{equ:stationary}
\sum_{\gamma\in\Gamma}\nu(\gamma)f_{\gamma}=1 \text{ on } \Lambda_{\Gamma}.
\end{equation}
Now we start to establish some properties of $f_{\gamma}$.

Recall that for two real functions $f$ and $g$, we write $f\ll g$ if there exists a constant $C>0$ only depending on the group $\Gamma$ such that $f\leq C g$. We write  $f\approx g$ if $f\ll g\ll f$.

\begin{lem}[Sullivan]\label{lem:shadow}
	There exists $C_1>0$ such that the following holds. For any $C\geq C_1$ there exists $C'$ such that for all $\gamma$ in $\Gamma$
	\begin{align*}
	\frac{1}{C'} r_{\gamma}^\delta\leq \mu(\cal O_{C}(o,\gamma^{-1}o))\leq C' r_{\gamma}^\delta.
	\end{align*}
\end{lem}
For the proof please see \cite[Page 10]{roblin2003ergodicite}.
\begin{lem}[Triangle rule]
	Let $ABC$ be a geodesic triangle in $X$. Let $\alpha,\beta,\gamma$ be the three angle of $A,B,C$ and let $a,b,c$ be the length of $BC,CA,AB$. Then 
	\[\frac{\sin\alpha}{\sinh a}=\frac{\sin\beta}{\sinh b}=\frac{\sin\gamma}{\sinh c}. \]
\end{lem}
See for example \cite[Page 148]{beardon1983discrete}. For $\xi$ in $\partial X$ and $t\in \R_+$, let $\xi_t$ be the point in the geodesic ray $o\xi$ with distance $t$ to $o$. We define distances on the boundary, that is the visual distance: for $x\in X$ and $\xi,\xi'$ on $\partial X$ $$d_x(\xi,\xi')=\lim_{t\rightarrow+\infty}e^{\frac{1}{2}(-d(x,\xi_t)-d(x,\xi'_t)+d(\xi_t,\xi'_t))}.$$
We fix the visual distance at the origin $o$, i.e. $d_o$, on the boundary $\partial X$.
\begin{lem}\label{lem:radius}
	The distance $d_o$ is the sinuous of the angle, that is for $\xi,\xi'$ in $\partial X$ we have $d_o(\xi,\xi')=sin\frac{1}{2}\angle\xi o\xi'$. The shadow $B_{\gamma}$ is of radius $\frac{\sinh C_0}{\sinh \kappa(\gamma)}$.
\end{lem}
\begin{proof}
	Let $p$ be the midpoint of $\xi_t,\xi'_t$. Then $op$ is orthogonal to the geodesic $\xi_t\xi'_t$. Let $\theta$ be the half of the angle $\xi o\xi'$. Hence by triangle rule
	\[\frac{\sin\pi/2}{\sinh d(o,\xi_t)}=\frac{\sin\theta}{\sinh (d(\xi_t,\xi'_t)/2)}. \]
	Therefore
	\[\sin\theta=\frac{\sinh (d(\xi_t,\xi'_t)/2)}{\sinh d(o,\xi_t)}. \]
	The function $\sinh s$ is almost $e^s/2$ when $s$ is large. When $t$ tends to infinite, we obtain
	\[\sin\theta=\lim_{t\rightarrow+\infty}e^{d(\xi_t,\xi'_t)/2-d(o,\xi_t)}=d_o(\xi,\xi'). \]
	
	Let $q$ be the tangent point of the ball $B(\gamma^{-1}o,C_0)$ with a geodesic ray starting from $o$. Then by triangle rule
	\[ \sin\angle qo(\gamma^{-1}o)=\sinh C_0/\sinh d(o,\gamma^{-1}o) .\]
	The proof is complete.
\end{proof}
We need a Lipschitz property of the Busemann function as a function on $\partial X$.
\begin{lem}\label{lem:coclip}
	For $\xi,\xi'$ in $\partial X$ and $x,y$ in $X$ with $d_x(\xi,\xi')\leq e^{-d(x,y)-2}$, we have
	\[|B_\xi(x,y)-B_{\xi'}(x,y)|\leq 32 e^{d(x,y)/2}d_x(\xi,\xi')^{1/2}. \]
\end{lem}
This lemma is implicitly contained in the proof of \cite[Proposition 3.5]{paulinequilibrium}. We summarize the properties of $f_{\gamma}$ in the following lemma
\begin{lem}\label{lem:fg} 
	Let $\gamma$ be an element in $\Gamma$.
	
	1)Let $\eta$ be a point in $B_{\gamma}$. Then we have 
	\begin{equation}\label{equ:yxgm}
	f_{\gamma}(\eta)\leq r_{\gamma}^{-\delta}=f_{\gamma}(\xmg)\leq e^{2C_0\delta} f_{\gamma}(\eta). 
	\end{equation}
	
	2)There exists $C_2>0$. If $\xi,\eta\in\partial X$ satisfy $d_o(\xi,\eta)\leq r_{\gamma}/e^2$, then we have
	\begin{equation}\label{equ:fglip}
	|f_ \gamma(\xi)/f_ \gamma(\eta)-1|\leq C_2 d_o(\xi,\eta)^{1/2}r_{\gamma}^{-1/2}.
	\end{equation}
	
	3)Let $\xi$ be a point in $\partial X$. Then we have
	\begin{equation}\label{equ:fgdec}
	f_ \gamma(\xi)\leq r_{\gamma}^\delta d_o(\xi,\xmg)^{-2\delta}.
	\end{equation}
\end{lem}

\begin{proof}%[Proof of Lemma \ref{lem:fg}] 
	For $\eqref{equ:yxgm}$, by definition of $\xmg$, we have
	\[f_{\gamma}(\xmg)=e^{-\delta B_{\xmg}(\gamma^{-1}o,o)}=e^{\delta d(\gamma^{-1}o,o)}=r_{\gamma}^{-\delta} .\]
	For $\eta$ in $B_{\gamma}$, by applying triangle inequality, we obtain \eqref{equ:yxgm}.
	
	For the second statement.
	By Lemma \ref{lem:coclip}, we get $$|B_\xi(\gamma^{-1}o,o)-B_{\eta}(\gamma^{-1}o,o)|\leq 32d_o(\xi,\eta)^{1/2}r_{\gamma}^{-1/2}\leq 32/e.$$ 
	Due to $|e^t-1|\ll |t|$ for $|t|\leq 32/e$, we obtain
	\begin{align*}
	|f_ \gamma(\xi)/f_ \gamma(\eta)-1|=|e^{-\delta(B_\xi(\gamma^{-1}o,o)-B_{\eta}(\gamma^{-1}o,o))}-1|%\ll \delta|B_\xi(\gamma^{-1}o,o)-B_{\eta}(\gamma^{-1}o,o)|
	\ll d_o(\xi,\eta)^{1/2}r_{\gamma}^{-1/2}.
	\end{align*}
	
	For the third statement, by definition
	\begin{align*}
	-B_\xi(\gamma^{-1}o,o)+d(\gamma^{-1}o,o)&=\lim_{z\rightarrow \xi}-d(z,\gamma^{-1}o)+d(z,o)+d(\gamma^{-1}o,o)\\
	&\leq \lim_{z\rightarrow \xi,w\rightarrow \xmg}-d(z,w)+d(z,o)+d(w,o)
	=-2\log d_o(\xi,\xmg).
	\end{align*}
	The proof is complete.
\end{proof}

For any $n$ in $\N$, let $r_n=e^{-4C_0n}$. We want to construct a cover of $\Lambda_{\Gamma}$. Let $S_n$ be the set of all $\gamma$ that satisfy
\begin{equation}\label{equ:sn}
e^{-2C_0}r_n\leq r_{\gamma}< r_n.
\end{equation} 
\begin{lem}\label{lem:cover}
	For any $n$ in $\N_0$, the family $\{B_{\gamma}\}_{\gamma\in S_n}$ consists of balls which cover $\Lambda_{\Gamma}$ with bounded Lebesgue number $C_3$, that is any $\xi\in\Lambda_{\Gamma}$ is contained in at most $C_3$ balls.
\end{lem}
\begin{proof}
	Let $\xi$ be a point in the limit set $\Lambda_{\Gamma}$, then the ray $o\xi$ is in the convex hull $hull(\Lambda_{\Gamma})$. Consider the point $p_n$ in the ray such that $d(p_n,o)=|\log r_n|+C_0$. Since the diameter of $C(\Gamma)$ is less than $C_0$ \eqref{equ:c0}, there exists $\gamma$ in $\Gamma$ such that 
	\begin{equation}\label{equ:xngo}
	d(p_n,\gamma^{-1}o)\leq C_0.
	\end{equation} 
	Hence $d(\gamma^{-1}o,o)\in [|\log r_n|,|\log r_n|+2C_0]$, which implies $\gamma\in S_n$. The inequality \eqref{equ:xngo} also implies that the distance from $\gamma^{-1}o$ to the ray $o\xi$ is less than $C_0$, i.e. $d(\gamma^{-1}o,o\xi)\leq C_0$. By the definition of shadow, we obtain $\xi\in B_{\gamma}$. The family $\{B_{\gamma}\}_{\gamma\in S_n}$ is a cover of the limit set $\Lambda_{\Gamma}$.
	
	It remains to prove that each point $\xi\in\Lambda_{\Gamma}$ is covered by a bounded number of balls. Let $q_n,q_n'$ be two points in the ray $o\xi$ with $d(q_n,o)=|\log r_n|-C_0$ and $d(q_n',o)=|\log r_n|+3C_0$. Let $J$ be the geodesic segment connecting $q_n$ and $q_n'$. Let 
	$$S_n(\xi)=\{\gamma\in S_n: \xi\in B_{\gamma} \}.$$
	Due to $\gamma\in S_{n}(\xi)$ and the definition of shadow, we obtain $d(\gamma^{-1}o,J)\leq C_0$, that is $\gamma^{-1}o$ are in $J^{C_0}$, the $C_0$ neighbourhood of $J$. The group $\Gamma$ is discrete without torsion, there exists $c>0$ such that $\min_{\gamma\neq e}d(o,\gamma^{-1}o)>c$. Then the set $\{\gamma^{-1}o \}_{\gamma\in S_n(\xi)}$ is a discrete set in $J^{C_0}$ and is $c$ separated, that is any two different points has distance greater than $c$. The volume of $J^{C_0}$ is uniformly bounded. Hence there is upper bound of the number of elements in $S_n(\xi)$. The proof is complete.
\end{proof}
\begin{rem}
	This is a key lemma where we need the hypothesis of convex cocompactness.
	When $\Gamma$ is a Schottky subgroup, the construction is easier. We can find a cover of the limit  set $\Lambda_{\Gamma}$ with no overlap.
\end{rem}
\subsection{Properties of operator $P_n$}
We will use the family of covers $S_n$ and basic properties of $f_{\gamma}$ to give properties of operator $P_n$, which will be constructed later.

Let $C_4=4C_0$ and let $\beta$, $\epsilon$ and $C_5$ be positive numbers defined subsequently such that
\begin{align*}
1-\beta>\beta+e^{-C_4},\ r_n^\epsilon=(1-\beta)^n,\ C_5=\frac{2 e^{\epsilon C_4}}{1-\beta}.
\end{align*}
For every $\gamma$ in $\Gamma$, Let $B_\gamma'=\cal O_{C_1}(o,\gamma^{-1}o)$.
By Lemma \ref{lem:shadow}, we know that $B_\gamma' \cap \Lambda_{\Gamma}\neq\emptyset$. We fix $\eta_\gamma$ in the set $B_\gamma' \cap \Lambda_{\Gamma}$ for each $\gamma$.
Let $R$ be a continuous function on $\partial X$, which is positive on $\limit$. For any $n$ in $\N$ and $\eta\in\partial X$, we define 
\begin{equation}
P_{n}R(\eta)=\sum_{ \gamma\in S_{n}} R(\eta_\gamma)r_ \gamma^\delta f_ \gamma(\eta).
\end{equation}
This construction of $P_nR$ will inherit the Lipschitz property of $f_{\gamma}$.
\begin{lem}\label{lem:pnlip}
	For any $n$ in $\N_0$, if $\xi,\eta\in\partial X$ satisfy $d_o(\xi,\eta)\leq r_{n+1}$, then
	\begin{align*}
	\left|\frac{P_nR(\xi)}{P_nR(\eta)}-1\right|\leq (d_o(\xi,\eta)/r_{n+1})^{1/2}.
	\end{align*}
\end{lem}
\begin{proof}
	For $\gamma\in S_n$, by \eqref{equ:sn} we have $d_o(\xi,\eta)\leq r_{n+1}=e^{-4C_0}r_n\leq e^{-2C_0}r_{\gamma}\leq r_{\gamma}/e^2$. By \eqref{equ:sn} and \eqref{equ:c0}, we get $r_{n+1}\leq r_\gamma e^{-2C_0}\leq r_\gamma/C_2^2$, then we use \eqref{equ:fglip} to obtain 
	\[\left|\frac{f_{\gamma}(\xi)}{f_{\gamma}(\eta)}-1\right|\leq C_2(d_o(\xi,\eta)/r_{\gamma})^{1/2}\leq (d_o(\xi,\eta)/r_{n+1})^{1/2}. \]
	As $P_{n}R$ is a positive linear combination of $f_{\gamma}$ with $\gamma\in S_{n}$, we have
	$|\frac{P_nR(\xi)}{P_nR(\eta)}-1|\leq (\frac{d_o(\xi,\eta)}{r_{n+1}})^{1/2}$.
\end{proof}
\begin{lem}\label{lem:pnr}
	For any $n$ in $\N$, if a positive function $R$ on $\limit$ satisfies the following condition:
	
	(1) For $\xi,\eta$ in $\Lambda_{\Gamma}$, if $d_o(\xi,\eta)\leq r_{n+1}$, then we have
	\begin{equation}\label{equ:rlip}
	|R(\xi)/R(\eta)-1|\leq (d_o(\xi,\eta)/r_{n+1})^{1/2}.
	\end{equation}
	
	(2) For $\xi,\eta$ in $\Lambda_{\Gamma}$, if $d_o(\xi,\eta)>r_{n+1}$, then
	\begin{equation}\label{equ:rnxy}
	|R(\xi)/R(\eta)|\leq C_5d_o(\xi,\eta)^\epsilon /(1-\beta)^n.
	\end{equation}
	Then there exist $C_6,C_7$ independent of $n, R$ such that for all $\eta\in\Lambda_{\Gamma}$
	\begin{equation}
	R(\eta)/C_6\leq P_{n+1}R(\eta)\leq C_7 R(\eta).
	\end{equation}
\end{lem}
\begin{proof}
	Since $\{B_{\gamma}\}_{\gamma\in S_{n+1}}$ is a cover of $\Lambda_{\Gamma}$, there is a $\gamma\in S_{n+1}$ such that $\eta\in B_{\gamma}$. By definition $P_{n+1}R(\eta)\geq R( \eta_\gamma)f_{ \gamma}(\eta)r_{\gamma}^\delta$. Thanks to $r_\gamma\leq r_{n+1}=e^{-4C_0(n+1)}\leq e^{-4}$, we get $\sinh \kappa(\gamma)\geq r_\gamma^{-1}/4$. Due to Lemma \ref{lem:radius} and \eqref{equ:sn}, we obtain $B_{\gamma}$ is of radius 
	$$\sinh C_0/\sinh \kappa(\gamma)\leq 2e^{C_0}r_{\gamma}\leq 2e^{C_0}r_{n+1}.$$ 
	Applying inequality \eqref{equ:rlip} or \eqref{equ:rnxy}  implies 
	\begin{equation}\label{equ:rxy}
	R(\eta_\gamma)\approx R(\eta). 
	\end{equation}
	Due to $\eta$ in $B_\gamma$, by \eqref{equ:yxgm}, we obtain $f_{\gamma}(\eta)\gg f_\gamma(\xmg)=r_{\gamma}^{-\delta}$. Putting it all together,
	we get $P_{n+1}R(\eta)\gg R(\eta)$.
	
	By Lemma \ref{lem:cover}, there is at most $C_3$ element $\gamma$ such that $B_{\gamma}$ contains $\eta$. For these $\gamma$, by \eqref{equ:rxy}, we have
	\begin{equation}\label{equ:gsn}
	\sum_{\gamma\in S_{n+1},\eta\in B_{\gamma}}R(\eta_\gamma)r_{\gamma}^\delta f_{\gamma}(\eta)\ll C_3R(\eta).
	\end{equation}
	For the rest of $\gamma$'s, recall that $B_{\gamma}'=\cal O_{C_1}(o,\gamma^{-1}o)$ is a smaller ball in $B_\gamma$. Due to Lemma \ref{lem:radius}, the radius of $B_\gamma'$ is $r(B_\gamma'):=\sinh C_1/\sinh \kappa(\gamma)$. 
	For $\gamma$ such that $\eta\notin B_{\gamma}$, we know there exists $C>0$ such that 
	\begin{equation}\label{equ:doeta}
	d_o(\eta,B_{\gamma}')>r_{n+1}/C.
	\end{equation}
	This is due to $r(B_\gamma)-r(B_\gamma')\gg r_\gamma\gg r_{n+1}$. By \eqref{equ:doeta} and \eqref{equ:rlip}, we know that even $d(\eta,\eta_\gamma)<r_{n+1}$ we also have $R(\eta_\gamma)\ll R(\eta) d_o(\eta,\eta_\gamma)^\epsilon/(1-\beta)^n$.
	Together with  \eqref{equ:rnxy} and \eqref{equ:fgdec}, we have
	\begin{equation}\label{equ:gnotsn}
	\begin{split}
	\sum_{\gamma\in S_{n+1},\eta\notin B_{\gamma}}R(\eta_\gamma)r_{\gamma}^\delta f_{\gamma}(\eta)
	&\ll R(\eta)(1-\beta)^{-n}\sum_{\gamma\in S_{n+1},\eta\notin B_{\gamma}}r_ \gamma^{\delta}f_{\gamma}(\eta)d_o(\eta,\eta_\gamma)^\epsilon\\
	&\leq R(\eta)(1-\beta)^{-n}\sum_{\gamma\in S_{n+1},\eta\notin B_{\gamma}}r_ \gamma^{2\delta} d_o(\eta,\eta_\gamma)^{\epsilon}d_o(\eta,\xmg)^{-2\delta}.
	\end{split}
	\end{equation}
	Due to \eqref{equ:c0}, we get $r(B_\gamma)-r(B_\gamma')\geq 4r(B_\gamma')$. This implies for $\xi$ in $B_\gamma'$, we have $d_o(\eta,\eta_\gamma)\geq d_o(\eta,\xi)-d_o(\xi,\eta_\gamma)\geq \frac{1}{2}d_o(\eta,\xi)$, which is also true if we replace $\eta_\gamma$ by $\xmg$. Together with Lemma \ref{lem:shadow}, that is  $r_{\gamma}^\delta\approx \mu(B_{\gamma}')$, and \eqref{equ:gsn}, \eqref{equ:gnotsn}, \eqref{equ:sn}, we obtain
	\begin{align*}
	P_{n+1}R(\eta)\ll R(\eta)\left(1+(1-\beta)^{-n}r_{n+1}^\delta\sum_{\gamma\in S_{n+1},\eta\notin B_{\gamma}}\int_{B_{\gamma}'}\frac{1}{d_o(\eta,\xi)^{2\delta-\epsilon}}\dd\mu(\xi)\right).
	\end{align*}  
	By Lemma \ref{lem:cover}, the union of balls $B_{\gamma}$ covers at most $C_3$ times, which is also true for smaller covers $B_{\gamma}'$. By \eqref{equ:doeta}, this implies
	\begin{align}\label{equ:pn+1}
	P_{n+1}R(\eta)\ll R(\eta)\left(1+(1-\beta)^{-n}r_{n+1}^\delta\int_{B(\eta,r_{n+1}/C)^c}\frac{1}{d_o(\eta,\xi)^{2\delta-\epsilon}}\dd\mu(\xi)\right).
	\end{align}  
	\begin{lem}\label{lem:intnu} Let $\theta$ be a positive number. For all $r>0$ and $\eta\in\Lambda_{\Gamma}$, we have
		\begin{equation}
		\int_{B(\eta,r)^c}\frac{1}{d_o(\eta,\xi)^{\delta+\theta}}\dd\mu(\xi)\ll_\theta \frac{1}{r^\theta}.\footnote
		{We write $f\ll_\theta g$ for two real functions if there exists a constant $C$ depending on the group and $\theta$ such that $f\leq C g$.}
		\end{equation}
	\end{lem}
	Therefore, Lemma \ref{lem:intnu} and \eqref{equ:pn+1} imply
	\begin{align*}
	P_{n+1}R(\eta)\ll R(\eta)(1+(1-\beta)^{-n}r_{n+1}^\delta/r_{n+1}^{\delta-\epsilon})=R(\eta)(1+(r_{n+1}/r_n)^\epsilon).
	\end{align*}
	The proof is complete.
\end{proof}
It remains to proof Lemma \ref{lem:intnu}.
\begin{proof}[Proof of Lemma \ref{lem:intnu}]
	Due to \cite[Theorem 7]{Sullivan2}, we have that $\mu(B(\eta,r))\ll r^\delta$ for all balls in $\partial X$ with $\eta\in\Lambda_{\Gamma}$ and $r>0$.
	
	Then
	\begin{align*}
	\int_{B(\eta,r)^c}&\frac{1}{d_o(\eta,\xi)^{\delta+\theta}}\dd\mu(\xi)=\sum_{1\leq n\leq 1/r}\int_{B(\eta,(n+1)r)-B(\eta,nr)}\frac{1}{d_o(\eta,\xi)^{\delta+\theta}}\dd\mu(\xi)\\
	&\leq\sum_{1\leq n\leq 1/r}\int_{B(\eta,(n+1)r)-B(\eta,nr)}\frac{1}{(nr)^{\delta+\theta}}\dd\mu(\xi)\\
	&\leq r^{-(\delta+\theta)}\left(\sum_{1\leq n\leq 1/r}\mu(B(\eta,(n+1)r))(\frac{1}{n^{\delta+\theta}}-\frac{1}{(n+1)^{\delta+\theta}})-\mu(B(\eta,r))\right)\\
	&\ll r^{-(\delta+\theta)}\left(\sum_{n\geq 1}((n+1)r)^\delta(\frac{1}{n^{\delta+\theta}}-\frac{1}{(n+1)^{\delta+\theta}})\right)\ll_\theta r^{-\theta}.
	\end{align*}
	The proof is complete.
\end{proof}
\subsection{Proof of Theorem \ref{thm:expmom}}
We start to prove our main theorem in this section.
We will construct $\{u_n\}_{n\in\bb N}$ by induction such that 
\[ |1-\sum_{n\leq M}u_n(\eta)|\rightarrow 0 \text{ as }M\rightarrow\infty \text{, uniformally for all }\eta\in\Lambda_{\Gamma},\] 
where $u_n$ is a finite linear combination of $f_{\gamma}$. The main idea is the same as that in \cite{connell2007harmonicity}. Once we have a function on $\Lambda(\Gamma)$ which satisfies the conditions in Lemma \ref{lem:pnr}, we can use the operator $P_{n+1}$ to drop some mass for elements in $S_{n+1}$.

Let $R_0=1$ be the constant function on $\partial X$. We now proceed by induction. For $n$ in $\N$, let
\begin{align*}
u_{n+1}=\frac{\beta}{C_7}P_{n+1}R_n,\ R_{n+1}=R_n-u_{n+1}.
\end{align*}
The following lemma is similar to \cite[Lemma 3]{lalley1986gibbs}.
\begin{lem}\label{lem:rnlip} 
	For any $n$ in $\N$, the following holds. The function $R_n$ is positive on $\Lambda_{\Gamma}$ and for $\xi,\eta$ in $\Lambda_{\Gamma}$, if $d_o(\xi,\eta)\leq r_{n+1}$, then  we have
	\begin{equation}\label{equ:rlip1}
	|R_n(\xi)/R_n(\eta)-1|\leq (d_o(\xi,\eta)/r_{n+1})^{1/2}.
	\end{equation}
	
	For $\xi,\eta$ in $\Lambda_{\Gamma}$, if $d_o(\xi,\eta)> r_{n+1}$, then
	\begin{equation}\label{equ:rnxy1}
	|R_n(\xi)/R_n(\eta)|\leq C_5d_o(\xi,\eta)^\epsilon /(1-\beta)^n.
	\end{equation}
\end{lem}
\begin{proof}
	The proof is by induction on $n$. For $n=0$, two inequalities hold trivially. 
	Suppose they hold for $n$, we will prove they also hold for $n+1$. By the induction hypothesis and Lemma \ref{lem:pnr}, we know 
	\begin{equation}\label{equ:un+1}
	u_{n+1}(\eta)\leq \beta R_n(\eta)\text{ for } \eta\in\Lambda_{\Gamma},
	\end{equation}
	which implies that $R_{n+1}$ is always a positive function on $\Lambda_{\Gamma}$.
	
	Due to Lemma \ref{lem:pnlip}, if $d_o(\xi,\eta)<r_{n+2}$, then 
	\begin{equation}\label{equ:un+2}
	|u_{n+1}(\xi)/u_{n+1}(\eta)-1|=|P_{n+1}R_n(\xi)/P_{n+1}R_n(\eta)-1|\leq (d_o(\xi,\eta)/r_{n+2})^{1/2}.
	\end{equation}
	Hence, for $\xi,\eta$ such that $d_o(\xi,\eta)<r_{n+2}$, by \eqref{equ:un+1}, \eqref{equ:rlip1}, \eqref{equ:un+2} we have
	\begin{align*}
	&\left|\frac{R_{n+1}(\xi)}{R_{n+1}(\eta)}-1\right|=\left|\frac{R_n(\xi)-u_{n+1}(\xi)}{R_n(\eta)-u_{n+1}(\eta)}-1\right|=\left|\frac{R_n(\xi)/R_n(\eta)-1}{1-u_{n+1}(\eta)/R_n(\eta)}-\frac{u_{n+1}(\xi)/u_{n+1}(\eta)-1}{R_n(\eta)/u_{n+1}(\eta)-1}\right|\\
	&\leq |R_n(\xi)/R_n(\eta)-1|/(1-\beta)+|u_{n+1}(\xi)/u_{n+1}(\eta)-1|/(1/\beta-1)\\
	&\leq \frac{d_o(\xi,\eta)^{1/2}}{r_{n+1}^{1/2}(1-\beta)}+\frac{d_o(\xi,\eta)^{1/2}}{(1/\beta-1)r_{n+2}^{1/2}}=\left(\frac{d_o(\xi,\eta)}{r_{n+2}}\right)^{1/2}(e^{-C_4/2}+\beta)/(1-\beta)\leq \left(\frac{d_o(\xi,\eta)}{r_{n+2}}\right)^{1/2}.
	\end{align*}	
	
	It remains to prove \eqref{equ:rnxy1}. By construction and \eqref{equ:un+1}, we have
	\begin{equation}\label{equ:rn+1}
	R_{n+1}(\xi)/R_{n+1}(\eta)=(R_{n}(\xi)-u_{n+1}(\xi))/(R_{n}(\eta)-u_{n+1}(\eta))\leq|R_{n}(\xi)/R_{n}(\eta)|/(1-\beta).
	\end{equation}
	If $d_o(\xi,\eta)\geq r_{n+1}$, then due to \eqref{equ:rn+1}, the inequality \eqref{equ:rnxy1} holds for $n+1$ is a direct consequence of case $n$. If else, we have $r_{n+2}<d_o(\xi,\eta)\leq r_{n+1}$. By \eqref{equ:rlip1} we have $R_{n}(\xi)/R_{n}(\eta)\leq 2$, then by \eqref{equ:rn+1}
	\begin{align*}
	R_{n+1}(\xi)/R_{n+1}(\eta)\leq 2/(1-\beta)= 2r_{n+1}^\epsilon/(1-\beta)^{n+2}\leq \frac{2 e^{\epsilon C_4}}{1-\beta}\frac{d_o(\xi,\eta)^\epsilon}{(1-\beta)^{n+1}}=\frac{C_5d_o(\xi,\eta)^\epsilon}{(1-\beta)^{n+1}}.
	\end{align*}
	The proof is complete.
\end{proof}
\begin{proof}[Proof of Theorem \ref{thm:expmom}]
	We start to prove our theorem. Let $C=C_6C_7$, where $C_6,C_7$ are constants in Lemma \ref{lem:pnr}. Lemma \ref{lem:rnlip} implies that the constructed $R_n$ is positive on $\Lambda_{\Gamma}$ and always satisfies the condition in Lemma \ref{lem:pnr} for $n\in\N$. Hence for a point $\eta$ in $\Lambda_{\Gamma}$, we apply Lemma \ref{lem:pnr} to obtain
	\begin{align*}
	R_{n+1}(\eta)= R_n(\eta)(1-u_{n+1}(\eta)/R_n(\eta))=R_n(\eta)\left(1-\frac{\beta P_{n+1}R_n(\eta)}{C_7 R_n(\eta)}\right)\leq R_n(\eta)(1-\beta/C).
	\end{align*}
	Iterating the above inequality, we get $R_n(\eta)\leq (1-\beta/C)^n$. Therefore, $R_n\rightarrow 0$ uniformly on $\Lambda_{\Gamma}$ as $n\rightarrow \infty$. 
	
	We set
	\begin{equation}\label{equ:mu}
	\nu(\gamma)=\begin{cases}
	R_{n-1}(\eta_\gamma)r_{\gamma}^\delta\beta/C_7&\text{ for } n\in\bb N_0, \gamma\in S_n, \\
	0 & \text{ for }\gamma\notin \cup_{n\in\bb N_0}S_n.
	\end{cases} 
	\end{equation}
	Then $R_{n}-R_{n+1}=\sum_{\gamma\in S_{n+1}}\nu(\gamma)f_{\gamma}$. It follows that $1=R_0=\sum_{\gamma\in\Gamma}\nu(\gamma)f_{\gamma}$ on $\Lambda_{\Gamma}$, which means that $\mu$ is $\nu$-stationary by \eqref{equ:stationary}.
	
	Next we verify the moment condition. Let $\epsilon_1$ be a positive number. Let $\|\gamma\|$ be the operator norm of its action on $\R^{n+1}$ equipped with euclidean norm. By the Cartan decomposition, we obtain $\|\gamma\|=r_\gamma^{-1}$ (see for example \cite[Remark 6.28 and Lemma 6.33]{benoistquint}). We can compute the exponential moment
	\begin{align*}
	\sum_{\gamma\in\Gamma}\nu(\gamma)\|\gamma\|^{\epsilon_1}= \sum_{n\in\bb N_0}\sum_{\gamma\in S_n}\nu(\gamma)\|\gamma\|^{\epsilon_1}\leq \frac{\beta}{C_7}\sum_{n\in\bb N_0}\sum_{\gamma\in S_n}R_{n-1}(\eta_\gamma)r_{\gamma}^{\delta-\epsilon_1}.
	\end{align*}
	While $r_{\gamma}^\delta\approx\mu(B_{\gamma})$ (Lemma \ref{lem:shadow}), we have $\sum_{\gamma\in S_n}r_{\gamma}^\delta\ll \sum_{\gamma\in S_n}\mu(B_{\gamma})=1$. Due to $r_{\gamma}\geq e^{-C_4(n+1)}$ and $R_n\leq (1-\beta/C)^n$, we get
	\begin{align}
	\sum_{\gamma\in\Gamma}\nu(\gamma)\|\gamma\|^{\epsilon_1}\ll\sum_{n\in\bb N}(1-\beta/C)^n e^{\epsilon_1 C_4(n+1)}.
	\end{align}
	Take $\epsilon_1$ small enough, the above sum is finite. 
	
	Lastly we prove $\Gamma_\nu=\Gamma$. Since the diameter of $C(\Gamma)$ is less than $C_0/2$, there exists $\gamma_1$ in $S_1$ such that $d(o,\gamma_1 o)\in[|\log r_1|+C_0/2,|\log r_1|+3C_0/2]$. By construction of $S_1$ \eqref{equ:sn} and $\eqref{equ:mu}$ we know that the set $\Gamma_{C_0}:=\{\gamma\in\Gamma| d(o,\gamma o)\leq C_0/2 \}$ is contained in $\gamma_1^{-1}S_1\subset S_1^{-1}S_1\subset \Gamma_\nu$.
	\begin{lem}\label{lem:generate}
		If $C_0$ is greater than 6 times the diameter of the quotient set $C(\Gamma)$, then the set $\Gamma_{C_0}$ generates the group $\Gamma$.
	\end{lem}
	By Lemma \ref{lem:generate}, the proof is complete.
\end{proof}
It remains to prove Lemma \ref{lem:generate}.
\begin{proof}[Proof of Lemma \ref{lem:generate}]
	This is a classical lemma. We give a proof here. Let $C_\Gamma$ be the diameter of the quotient $C(\Gamma)$. For any $\gamma$ in $\Gamma$, we will find a sequence $\beta_j$, $0\leq j\leq k$ in $\Gamma_{C_0}$ such that $\gamma=\beta_0\cdots \beta_k$, which finishes the proof.
	
	In the geodesic $o(\gamma o)$, let $p_j$ be the point with distance $jC_\Gamma$ to $o$. Suppose that $kC_\Gamma\leq d(o,\gamma o)<(k+1)C_\Gamma$ and let $p_{k+1}=\gamma o$. Since $o(\gamma o)$ is in the convex hull, for every $p_j$ with $1\leq j\leq k$, by the definition of $C_\Gamma$, we can find $\gamma_j$ in $\Gamma$ such that $d(\gamma_jo,p_j)\leq C_\Gamma$. Let $\gamma_0=e$ and $\gamma_{k+1}=\gamma$. Hence for $0\leq j\leq k$
	$$d(\gamma_j o,\gamma_{j+1}o)\leq d(\gamma_jo,p_j)+d(p_j,p_{j+1})+d(p_{j+1},\gamma_{j+1}o)\leq 3C_\Gamma.$$
	Let $\beta_j=\gamma_{j}^{-1}\gamma_{j+1}$. Then $d(\beta_jo,o)=d(\gamma_{j}^{-1}\gamma_{j+1}o,o)=d(\gamma_{j+1} o,\gamma_{j} o)\leq 3C_\Gamma\leq C_0$, which implies $\beta_j\in \Gamma_{C_0}$. Therefore $\gamma=\beta_0\cdots \beta_k$.
\end{proof}

 \noindent {\it Jialun Li\\
IMB, Universit\'e de Bordeaux, 
351 cours de la Lib\'eration, 33405 Talence, France \\
email: {\tt jialun.li@math.u-bordeaux.fr}   
  
\noindent {\it Fr\'ed\'eric Naud\\
LMA, Avignon Universit\'e, 301 rue Baruch de Spinoza, 84916 Avignon Cedex, France\\
email: {\tt frederic.naud@univ-avignon.fr}   
   
\noindent {\it Wenyu Pan\\
Penn State University, State College, PA 16802, USA\\
email: {\tt wup60@psu.edu}

\end{document}